\documentclass[11pt,a4paper,reqno]{amsart}
\title{There are only countably many locally tabular bi-intermediate logics of co-trees}

\author{Miguel Martins}
\address{\textbf{Miguel Martins:} Departament de Filosofia\\
Facultat de Filosofia\\
Universitat de Barcelona (UB)\\
Carrer Montalegre, 6, 08001 Barcelona, Spain}
\email{miguelplmartins561@gmail.com}

\date{}

\usepackage{mathpazo}  
\usepackage{graphicx}            

\usepackage[utf8]{inputenc}     
\usepackage{amssymb, latexsym, stmaryrd, dsfont, amsmath, amsthm, amsfonts, mathrsfs, amsbsy, mathrsfs, mathtools}            
\usepackage[bottom,symbol]{footmisc}
\usepackage{appendix, verbatim}
\usepackage{etoolbox,graphicx,color}
\usepackage{amscd}
\usepackage{tabularx}
\usepackage{epsfig}
\usepackage{url}
\usepackage{color}
    \usepackage[all, knot]{xy}
        \xyoption{arc}
        \xyoption{web}
\usepackage{tikz}
\usepackage{multicol}
\usepackage[justification=centering]{caption}
\usetikzlibrary{calc}

\allowdisplaybreaks[1]							

\usepackage[margin=1in]{geometry}

\usepackage{sidecap}                   
\usepackage{bm}                          
\usepackage{enumerate}                   
\usepackage{microtype}  

\makeatletter                            
\def\MT@register@subst@font{\MT@exp@one@n\MT@in@clist\font@name\MT@font@list
 \ifMT@inlist@\else\xdef\MT@font@list{\MT@font@list\font@name,}\fi}
\makeatother
\usepackage[pdftex,bookmarks,bookmarksnumbered,linktocpage,  
         colorlinks,linkcolor=blue,citecolor=blue]{hyperref}
\usepackage{memhfixc}

\usepackage{color}

\definecolor{Salmon}{RGB}{250,128,114}
\definecolor{Crimson}{RGB}{220,20,60}
\definecolor{DarkOrange}{RGB}{255,140,0}
\definecolor{Khaki}{RGB}{240,230,140}
\definecolor{GreenYellow}{RGB}{173,255,47}
\definecolor{MediumSeaGreen}{RGB}{60,179,113}
\definecolor{OliveDrab}{RGB}{107,142,35}
\definecolor{LightSeaGreen}{RGB}{32,178,170}
\definecolor{Aquamarine}{RGB}{127,255,212}
\definecolor{SteelBlue}{RGB}{70,130,180}
\definecolor{Navy}{RGB}{0,0,128}
\definecolor{Purple}{RGB}{128,0,128}
\definecolor{Orchid}{RGB}{218,112,214}
\definecolor{Brown}{RGB}{165,42,42}
\definecolor{Chocolate}{RGB}{210,105,30}
\definecolor{SandyBrown}{RGB}{244,164,96}

\usepackage{tikz}                            
\usetikzlibrary{positioning}
\usetikzlibrary{automata}
 
\usepackage{hyperref}
\hypersetup{colorlinks,citecolor= blue,urlcolor=blue}

\newtheorem{Theorem}{Theorem}[section]
\newtheorem{Lemma}[Theorem]{Lemma}
\newtheorem{Proposition}[Theorem]{Proposition}

\theoremstyle{definition}
\newtheorem{defi}[Theorem]{Definition}
\newtheorem{exa}[Theorem]{\textbf{Example}}

\theoremstyle{remark}
\newtheorem{Remark}[Theorem]{Remark}

\newtheorem{exer*}[Theorem]{Exercise*}

\makeatletter
\newcommand\niton{\mathrel{\m@th\mathpalette\canc@l\owns}}
\newcommand\canc@l[2]{{\ooalign{$\hfil#1/\mkern1mu\hfil$\crcr$#1#2$}}}
\newcommand{\bicat}[0]{\operatorname{\mathsf{bi-HA}}}

\newcommand{\bacat}[0]{\mathsf{BA}}
\newcommand{\J}[0]{\mathcal{J}}

\newcommand{\bipc}[0]{\operatorname{\mathsf{bi-IPC}}}
\newcommand{\bilc}[0]{\operatorname{\mathsf{bi-LC}}}
\newcommand{\ipc}[0]{\mathsf{IPC}}
\newcommand{\cpc}[0]{\mathsf{CPC}}
\newcommand{\bg}[0]{\operatorname{\mathsf{bi-GA}}}

\newcommand{\lc}[0]{\operatorname{\mathsf{bi-GD}}}
\newcommand{\bigd}[0]{\operatorname{\mathsf{bi-GD}}}

\newcommand{\X}[0]{\mathcal{X}}

\newcommand{\T}[0]{\mathcal{T}}
\newcommand{\Y}[0]{\mathcal{Y}}
\newcommand{\V}[0]{\mathsf{V}}

\newcommand{\bivar}[0]{\operatorname{\mathsf{bi-HA}}}

\newcommand{\A}[0]{\mathbf{A}}
\newcommand{\B}[0]{\mathbf{B}}

\newcommand{\La}[0]{\mathfrak{L}}
\newcommand{\F}[0]{\mathfrak{F}}
\newcommand{\M}[0]{\mathfrak{M}}

\newcommand{\C}[0]{\mathfrak{C}}
\newcommand{\down}[0]{{\downarrow}}
\newcommand{\up}[0]{{\uparrow}}
\newcommand{\SSS}{\mathbb{S}}
\newcommand{\HHH}{\mathbb{H}}
\newcommand{\PPP}{\mathbb{P}}
\newcommand{\ZZ}{\mathbb{Z}}
\newcommand{\III}{\mathbb{I}}
\newcommand{\lang}[1][]{\mathcal L_{#1}}

\newcommand{\balg}[0]{\text{bi-G\"odel algebra}}
\newcommand{\seq}{(p_i)_{i \in \omega}}
\newcommand{\form}{\mathbf{Fm}}
\makeatother

\let\leq=\leqslant
\let\nleq=\nleqslant
\let\geq=\geqslant

\newcommand{\bit}{\begin{itemize}}    
\newcommand{\eit}{\end{itemize}}
\newcommand{\ben}{\begin{enumerate}}
\newcommand{\een}{\end{enumerate}}
\newcommand{\benormal}{\ben[\normalfont 1.]}   
\let\enormal\een
\newcommand{\benroman}{\ben[\normalfont (i)]}  
\let\eroman\een
\newcommand{\benbullet}{\ben[\textbullet]}     
\let\ebullet\een
\newcommand{\vs}{\vspace{.3cm}}

\DeclareRobustCommand{\upcc}{%
  \mathord{\vphantom{\uparrow}\text{%
    \ooalign{%
      $\uparrow$\kern-.515em\raisebox{-1.25ex}{\scalebox{0.88}{$\circ$}}\cr
    }%
  }}%
}
\newcommand{\upc}{\ensuremath \raisebox{1.5pt}{$\upcc$}}

\DeclareRobustCommand{\downcc}{%
  \mathord{\vphantom{\downarrow}\text{%
    \ooalign{%
      $\downarrow$\kern-.515em\raisebox{1.35ex}{\scalebox{0.88}{$\circ$}}\cr
    }%
  }}%
}
\newcommand{\downc}{\ensuremath \raisebox{-1.5pt}{$\downcc$}}

\newcommand{\precc}{{\prec}}

\newcommand{\qq}{^}
\newcommand{\ccpc}{\vdash_{\cpc}}
\newcommand{\cbipc}{{\vdash_{\bipc}}}

\newcommand{\cipc}{{\vdash_{\ipc}}}
\newcommand{\cbigd}{{\vdash_{\lc}}}
\newcommand{\vfc}{{\V_{FC}}}
\newcommand{\cfc}{{\vdash_{FC}}}

\begin{document}

\begin{abstract}
A bi-Heyting algebra validates the Gödel-Dummett axiom $(p \to q) \lor (q \to p)$ iff the poset of its prime filters is a disjoint union of co-trees. 
Bi-Heyting algebras of this kind are called bi-Gödel algebras and form a variety $\bg$ that algebraizes the extension $\bigd$ of bi-intuitionistic logic axiomatized by the Gödel-Dummett axiom. 
In this paper we show that there are only countably many locally tabular bi-intermediate logics of co-trees, all of which are finitely axiomatizable.

The theory of canonical formulas of bi-Gödel algebras has shown that $\bg$ has continuum many subvarieties, among which the locally finite ones coincide with the subvarieties of the $\V_n\coloneqq \{\A \in \bg \colon \A \models \beta(\C_n)\}$ (where $\beta(\C_n)$ is the subframe formula  of the $n$-comb).
We identify the multiset projectivity relation (a binary relation that, when defined on the set of finite multisets of a better partial order, is necessarily a better partial order) and use it to prove that every $\V_n$ is a Specht variety, hence has only countably many subvarieties, all of which are finitely axiomatizable.
By the algebraizability of $\bigd$, the main result follows.
We also provide an informative depiction of the lattice of varieties of bi-Gödel algebras.
\end{abstract}
\maketitle

\section{Introduction}

The \textit{bi-intuitionistic propositional calculus} $\cbipc$ is the conservative extension of the \textit{intuitionistic propositional calculus} $\cipc$ obtained by introducing $\gets$ to the language, a binary connective which behaves dually to $\to$ and is known as \textit{co-implication} (also called \textit{exclusion} or \textit{subtraction}).
To get an intuition for the behavior of the co-implication, we can utilize the Kripke semantics of $\cbipc$ \cite{Rauszer6}: if $x$ is a point in a Kripke model $\M$, then for formulas $\phi$ and $\psi$ we have
\[
\M, x \models \phi \gets \psi \text{ iff } \exists y \leq x \; (\M,y \models \phi \text{ and } \M,y\not \models \psi).
\]  

The co-implication $\gets$ gives $\cbipc$ significantly greater expressive power than that of $\cipc$.
This is witnessed, for instance, in \cite{Wolter}, where Gödel's embedding of $\cipc$ into the modal logic $\mathsf{S4}$ is extended to an embedding of $\cbipc$ into tense-$\mathsf{S4}$, and a version of the Blok-Esakia Theorem \cite{Blok76,Esakia76} is proved, showing that the lattice $\Lambda(\bipc)$ of \textit{bi-intermediate logics} (consistent axiomatic extensions\footnote{From now on we will use \textit{extension} as a synonym of \textit{axiomatic extension}.} of $\bipc$) is isomorphic to that of consistent normal tense logics containing $\mathsf{Grz.t}$ (see also \cite{Cleani21,Stronkowski}).

Furthermore, the addition of this new connective endows $\cbipc$ with a symmetry which is notably absent in $\cipc$, since now each connective $\land, \to, \bot$ has its dual $\lor, \gets, \top$, respectively.
This greater symmetry is reflected in the fact that $\cbipc$ is algebraized in the sense of \cite{BP89} by the variety $\bicat$ of \textit{bi-Heyting algebras} \cite{Rauszer3} (Heyting algebras whose order duals are also Heyting algebras). Consequently, the lattice $\Lambda(\bipc)$ is dually isomorphic to that of nontrivial varieties of bi-Heyting algebras. The latter, in turn, is amenable to the methods of universal algebra and duality theory since the category of bi-Heyting algebras is dually equivalent to that of \textit{bi-Esakia spaces} \cite{Esakia2} (see also \cite{Bezhan1}).

Motivated by their connection with bi-intuitionistic logic, the theory of bi-Heyting algebras was developed by Rauszer and other in a series of papers (see, e.g., \cite{Beazer,Kohler1,Rauszer2,Rauszer3,Rauszer6,Sanka1}).
However, bi-Heyting algebras also arise naturally in other fields of research.
For example, the lattice of open sets of an Alexandrov space is always a bi-Heyting algebra, and so is the lattice of subgraphs of an arbitrary graph (see, e.g., \cite{Taylor}). 
Similarly, every quantum system can be associated with a complete bi-Heyting algebra \cite{Doring}. 
Many other examples can be found, especially in the field of topos theory \cite{Lawvere1,Lawvere2,Reyes}. 

The thorough investigation of the lattice of \textit{intermediate logics} (consistent extensions of $\cipc$) was a very fruitful topic in nonclassical logic (see, e.g., \cite{Zakha}), but currently $\Lambda(\bipc)$ lacks such an in-depth analysis (for some recent developments in the study of $\cbipc$, see, e.g.,  \cite{Badia,BJib,Gore1,Gore2,Shramko}).
In \cite{Paper1}, the authors contributed to the investigation of $\Lambda(\bipc)$ by focusing on the sublattice $\Lambda(\lc)$ of consistent extensions of the \textit{bi-intuitionistic Gödel-Dummett logic}
\[
\cbigd \coloneqq \cbipc + (p\to q) \lor (q \to p).
\]

The formula $(p\to q) \lor (q \to p)$ is called the 
\textit{prelinearity axiom} (or the \textit{Gödel-Dummett axiom}) and over $\cipc$ it axiomatizes the \textit{intuitionistic linear calculus} $\vdash_{\mathsf{LC}}$ (or the intuitionistic \textit{Gödel-Dummett logic}).
$\vdash_{\mathsf{LC}}$ has been widely studied (see, e.g., \cite{Dummett,Goedel,Horn1,Horn}), and is well-known to be both the intermediate logic of chains (in the sense that it is Kripke complete with respect to the class of \textit{chains}, i.e., linearly ordered Kripke frames) and the intermediate logic of \textit{co-trees} (Kripke frames with a greatest element and whose principal upsets are linearly ordered).
In contrast, while $\cbigd$ is the bi-intermediate logic of co-trees, it is proved in \cite[Thm. 3.10]{Paper1} that the bi-intermediate logic of chains is its proper extension 
\[
\vdash_{\operatorname{\mathsf{bi-LC}}} \coloneqq \cbipc + (p \to q) \lor (q \to p) + \neg[ (q\gets p) \land (p \gets q)],
\]
there called the \textit{bi-intuitionistic linear calculus} (see also \cite[Thm. 4.25]{Paper1} for a different axiomatization of $\vdash_{\operatorname{\mathsf{bi-LC}}}$).
This suggests that the language of $\cbipc$ is more appropriate to study tree-like structures than that of $\ipc$, since $\cbipc$ is capable of distinguishing the class of chains from that of co-trees, while $\cipc$ cannot. 
Yet another example of this is that $\cipc$ is well known to be both the intermediate logic of Kripke frames and that of \textit{trees} (order duals of co-trees), and while $\cbipc$ is the bi-intermediate logic of Kripke frames \cite{Rauszer6}, it is also shown in \cite[Thm. 3.10]{Paper1} that the bi-intermediate logic of trees is
\[
\vdash_{\operatorname{\mathsf{bi-GD}}\qq \partial} \coloneqq \cbipc + \neg[ (q\gets p) \land (p \gets q)].
\]

Notably, because of the symmetric nature of bi-intuitionistic logic, all of the results in \cite{Paper1,Martins01,MARTINS2025103563} and in this current paper about extensions of the bi-intermediate logic of co-trees $\cbigd$ can be rephrased in a straightforward manner as results on extensions of the bi-intermediate logic of trees $\vdash_{\operatorname{\mathsf{bi-GD}}\qq \partial}$, by replacing every occurring formula $\varphi$ by its dual $\lnot \varphi^\partial$ ($\varphi^\partial$ is the formula obtained from $\varphi$ by replacing each occurrence of $\land, \lor, \alpha\to \beta, \alpha\gets \beta, \bot, \top$ in $\varphi$ by $\lor, \land, \beta \gets \alpha, \beta \to \alpha, \top, \bot$ respectively) and every algebra or Kripke frame by its order dual. 

There are also other properties of $\cbigd$ that diverge significantly from those of its intuitionistic fragment $\vdash_{\mathsf{LC}}$.
For example, it is known that the lattice $\Lambda(\mathsf{LC})$ of consistent extensions of $\vdash_{\mathsf{LC}}$ is a chain of order type $(\omega+1)^\partial$ (see, e.g., \cite{Zakha}), whereas it is shown in \cite[Thm. 4.16]{Paper1} that the lattice $\Lambda(\lc)$ is not a chain and has the cardinality of the continuum. 
It is also well known that $\vdash_{\mathsf{LC}}$ is \textit{locally tabular} \cite{Horn} (a logic is said to be \textit{locally tabular} when there are only finitely many formulas (up to logical equivalence) in each finite number of variables), but it is an immediate consequence of \cite[Cor. 5.31]{Paper1} that $\cbigd$ is not.

Local tabularity in the setting of $\Lambda(\bigd)$ has been studied in \cite{Paper1,Martins01,MARTINS2025103563}, where it shown to be intrinsically connected to the \textit{finite combs}, a type of finite co-trees depicted in Figure \ref{Fig:finite-combs}.
The variety of bi-Gödel algebras generated by the duals of the finite combs and its bi-intermediate logic (called the \textit{logic of the finite combs}) are denoted by $\vfc \coloneqq \HHH \SSS \PPP\{\C_n^* \colon n \in \ZZ^+\}$ and $\cfc$, respectively.
\begin{figure}[h]
\begin{tikzpicture}
    \tikzstyle{point} = [shape=circle, thick, draw=black, fill=black , scale=0.35]
    \node [label=right:{$c_1'$}] (1') at (1,0) [point] {};
    \node [label=left:{$c_1$}] (1) at (0.5,0.5) [point] {};
    \node [label=right:{$c_2'$}] (2') at (1.5,.5) [point] {};
    \node [label=left:{$c_2$}] (2) at (1,1) [point] {};
    \node [label=above:{$c_n$}] (n) at (1.75,1.75) [point] {};
    \node [label=right:{$c_n'$}] (n') at (2.25,1.25) [point] {};
    
    \draw (1)--(2);
    \draw (1')--(1);
    \draw (2')--(2);
    \draw (n')--(n);
    \draw [dotted] (2)--(n);
\end{tikzpicture}
\caption{The $n$-comb $\C_n$, where $n \in \ZZ^+$.}
\label{Fig:finite-combs}
\end{figure}
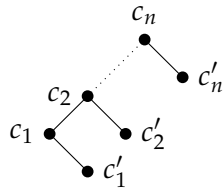 

The theories of Jankov and subframe formulas of bi-Gödel algebras (which were developed in \cite{Paper1, Martins01}, but for an overview of these formulas and their use in superintuitionistic and modal logics we refer to \cite{Bezhan2} and \cite{Zakha}, respectively) provide sufficient machinery to make the aforementioned connection between local tabularity and the finite combs apparent.
This is because the validation of these types of formulas in a bi-Gödel algebra $\A$ yields restrictions on the poset structure of the bi-Esakia dual of $\A$.

For example, if we denote the subframe formula of (the algebraic dual of) the $n$-comb by $\beta(\C_n)$, then $\A \models \beta(\C_n)$ iff $\C_n$ does not order embed into the bi-Esakia dual of $\A$ iff $\C_n \not \hookrightarrow \A_*$.
This is presented more generally in Lemma \ref{subframe lemma}, but see also \cite[Lem.~4.24]{Paper1}.
Similarly, we denote the Jankov formula of (the algebraic dual of) the $n$-comb by $\J(\C_n)$ and refer to \cite[Lem.~4.9]{Paper1} for a Jankov Lemma.

The defining properties of these formulas were essential to the derivation of the following criterion for local tabularity in $\Lambda(\bigd)$:

\begin{Theorem}{\cite[Cor.~5.31]{Paper1}}\label{big crit}
    Let ${\vdash} \in \Lambda(\bigd)$ and $\V_\vdash$ be its variety of bi-Gödel algebras.
    The following conditions are equivalent:
\begin{multicols}{2}
    \benroman
        \item $\vdash \textit{ is locally tabular}$;
        \item ${\vdash} \nsubseteq \cfc$;
        \item $\vdash \J(\C_n) \textit{ for some } n \in \ZZ^+$;
        \item $\vdash \beta(\C_n) \textit{ for some } n \in \ZZ^+$;
        \item $\V_\vdash$ is locally finite;
        \item $\vfc \nsubseteq \V_\vdash$;
        \item $\V_\vdash \textit{ omits the algebraic dual of a finite comb}$;
        \item $\exists n \in \ZZ^+, \, \forall \A \in \V_\vdash \, \big(\C_n \not \hookrightarrow \A_*\big)$.
    \eroman
\end{multicols}
\end{Theorem}

\noindent The last condition of the theorem can be rephrased as ``there exists a natural bound for the size of the finite combs that can be order embedded into the bi-Esakia models of $\V_\vdash$''.
We also note that it is immediate form the criterion that $\cfc$ is the only \textit{pre-locally tabular} logic in $\Lambda(\bigd)$, that is, $\cfc$ is not locally tabular but all of its proper  extensions are so.

Subsequently, a finite axiomatization for the logic of the finite combs $\cfc$ was found in \cite{MARTINS2025103563}.
And since $\cfc$ has the finite model property by definition, it follows from that this logic is decidable.

Together with the above criterion, in particular, with the fact that an extension of $\cbigd$ is locally tabular iff it is not a sublogic of $\cfc$, the decidability of $\cfc$ ensures that local tabularity is decidable in $\Lambda(\bigd)$.
This is formalized bellow.

\begin{Theorem} {\cite[Thm.~3.3]{MARTINS2025103563}}
    The logic $\cfc$ of the finite combs is decidable and coincides with
    \[
    \cbigd + \beta (\F_0) + \J (\F_1) +\J (\F_2) + \J (\F_3),
    \]
    where $\F_0, \F_1, \F_2,\F_3$ are the finite co-trees depicted in Figure \ref{Fig:the co-trees}.
    
    Consequently, the problem of determining if a finitely axiomatizable  extension of $\cbigd$ is locally tabular is decidable.
\end{Theorem}

\begin{figure}[h]
\centering
\begin{tabular}{cccc}
\begin{tikzpicture}
    \tikzstyle{point} = [shape=circle, thick, draw=black, fill=black , scale=0.35]

    \node [label=left:{$d$}] (d) at (-.5,-.25) [point] {};
    \node [label=right:{$e$}] (e) at (.5,-.25) [point] {};
    \node [label=left:{$b$}] (b) at (-.5,.5) [point] {};
    \node [label=right:{$c$}] (c) at (.5,.5) [point] {};
    \node [label=above:{$a$}] (a) at (0,1) [point] {};
    \node [label=below:{\Large{$\F_0$}}] at (0,-.75) [] {};

    \draw (d)--(b)--(a)--(c)--(e);
\end{tikzpicture}
\hspace{1cm}
\begin{tikzpicture}
    \tikzstyle{point} = [shape=circle, thick, draw=black, fill=black , scale=0.35]
    \node [label=left:{$a$}] (a) at (0,1) [point] {};
    \node [label=left:{$b$}] (b) at (0,0) [point] {};
    \node [label=left:{$c$}] (c) at (0,-1) [point] {};
    \node [label=below:{\Large{$\F_1$}}] at (0,-1.5) [] {};

    \draw (c)--(b)--(a);
\end{tikzpicture}
\hspace{1cm}
\begin{tikzpicture}
    \tikzstyle{point} = [shape=circle, thick, draw=black, fill=black , scale=0.35]
    \node [label=above:{$d$}] (d) at (0,0) [point] {};
    \node [label=above:{$c$}] (c) at (0.5,0.5) [point] {};
    \node [label=above:{$b$}] (b) at (01,01) [point] {};
    \node [label=above:{$a$}] (a) at (01.5,01.5) [point] {};
    \node [label=right:{$a'$}] (a') at (2,01) [point] {};
    \node [label=below:{\Large{$\F_2$}}] at (1.2,-.5) [] {};

    \draw (d)--(a)--(a');

\end{tikzpicture}
\hspace{1cm}
\begin{tikzpicture}
    \tikzstyle{point} = [shape=circle, thick, draw=black, fill=black , scale=0.35]

    \node [label=left:{$b$}] (b) at (-.75,0) [point] {};
    \node [label=left:{$c$}] (c) at (0,0) [point] {};
    \node [label=left:{$d$}] (d) at (.75,0) [point] {};
    \node [label=above:{$a$}] (a) at (0,1) [point] {};
    \node [label=below:{\Large{$\F_3$}}] at (0,-.5) [] {};

    \draw (b)--(a)--(c);
    \draw (d)--(a);
\end{tikzpicture}

\end{tabular}
\caption{The co-trees $\F_0$, $\F_1$, $\F_2,\text{ and }\F_3$.}
\label{Fig:the co-trees}
\end{figure}
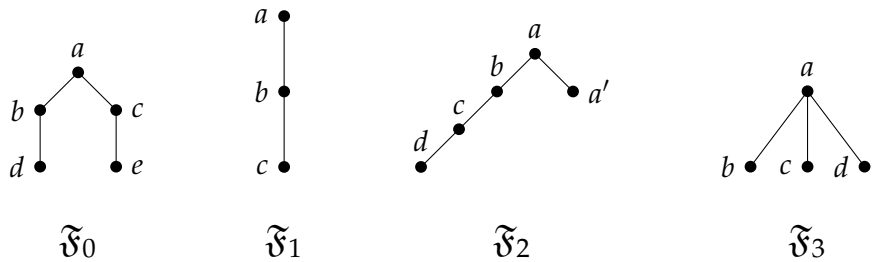

In this current paper, we contribute to the above discussion by proving in Theorem \ref{main thm} that:
\begin{itemize}
    \item[(C1)] there are only $\aleph_0$ locally tabular extensions of $\cbigd$,
        \vspace{.2cm}
    \item[(C2)] every locally tabular extension of $\cbigd$ is finitely axiomatizable.
\end{itemize}
We recall that $\Lambda(\bigd)$ has the size of the continuum and highlight the sharp contrast with the intuitionistic case, where it is well known that $\vdash_{\mathsf{LC}}$ has $\aleph_0$ extensions, all of which are locally tabular.

Our proof of (C1\&2) begins by noting that Theorem \ref{big crit}, together with the defining property of subframe formulas, entails that an element of $\Lambda(\bigd)$ is locally tabular exactly when it is the bi-intermediate logic of a subvariety of
\[
\V_n \coloneqq \{\A \in \bg\colon \A \models \beta(\C_n)\} = \{\A \in \bg\colon \C_n \not \hookrightarrow \A_*\},
\]
for some positive integer $n$.
Consequently, we can derive (C1) from the equality
\begin{equation}\label{the equality}
    |\bigcup_{n \in \mathbb{Z}^+} \Lambda(\V_n)|=\aleph_0,
\end{equation}
where $\Lambda(\V_n)$ denotes the lattice of nontrivial subvarieties of $\V_n$.

That $\aleph_0 \leq |\bigcup_{n \in \omega} \Lambda(\V_n)|$ follows readily from the simple observation (detailed in Lemma \ref{lem antichains equivalence}, but a glance at Figure \ref{Fig:finite-combs} should be convincing enough) that given positive integers $n<m$, then $\C_n \hookrightarrow \C_m$ but $\C_m \not \hookrightarrow \C_n$, and therefore $\V_n \subsetneq \V_m$.

To prove the reverse inequality, it suffices to show that every $\Lambda(\V_n)$ is at most countable, as this implies
\[
|\bigcup_{n \in \mathbb{Z}^+} \Lambda(\V_n)| \leq \sum_{n \in \mathbb{Z}^+} |\Lambda(\V_n)| \leq \sum_{n \in \mathbb{Z}^+} \aleph_0 = \aleph_0.
\]
We achieve this by relying on two distinct concepts:

\benroman
    \item \underline{\textbf{Specht varieties}} 

    A variety $\V$ is \textit{Specht} when $\V$ and all of its subvarieties are finitely axiomatizable.
 
    In \cite[Thm.~6.21]{CITKIN_2020}, some conditions on a variety $\V$ are shown to be equivalent to this definition, in the case that $\V$ is locally finite, finitely axiomatizable, and congruence distributive (three properties that every $\V_n$ satisfies).
    We state this result in Theorem \ref{specht} for convenience.
    
    In particular, the theorem ensures that for such a variety $\V$, we have that $\V$ is Specht iff $|\Lambda(\V)|\leq \aleph_0$.
    Thus, if we can show that every $\V_n$ satisfies one of the other conditions stated in the theorem, not only do we obtain that $\Lambda(\V_n)$ is at most countable (hence finishing the proof of (\ref{the equality}), and consequently of (C1)), but we also get (C2) by the definition of Specht varieties (recall that every locally finite variety of bi-Gödel algebras must be a subvariety of some $\V_n)$.

    Accordingly, the technical part of this paper is dedicated to proving that every $\V_n$ satisfies condition (iii) of Theorem \ref{specht}.
    Using the (finite) bi-Esakia duality for bi-Gödel algebras (see Theorem \ref{fin duality}) and some notable properties of the variety $\bg$ (see Theorem \ref{props of bg}), we show at the start of Section \ref{Sec counting} that $\V_n$ satisfies Theorem \ref{specht}.(iii) iff
    \[
    \T_n \coloneqq \{\X \colon \X \text{ is a finite co-tree s.t. } \C_n \not \hookrightarrow \X \} 
    \]
     has no infinite antichain w.r.t. $\leq_p$, a partial order called \textit{bi-p-morphic image relation} and defined on the class of finite co-trees by: $\X \leq_p \Y$ iff there exists a surjective bi-p-morphism (see Definition \ref{def bi-p-morphism}) from $\Y$ onto $\X$.

     \item \underline{\textbf{Better partial orders}} 

    We aim to show that the posets $(\T_n, \leq_p)$ cannot contain infinite antichains. 
    Even when considering the whole family of finite co-trees, the definition of $\leq_p$ already forbids the existence of infinite descending chains.
    Hence our goal can be restated as: every $(\T_n, \leq_p)$ is a \textit{well partial order} (WPO for short), i.e., a poset without infinite descending chains nor infinite antichains.

     That $(\T_1, \leq_p)$ is a WPO is trivial (since $\C_1$ is a two-element chain, $\T_1$ contains only one member, the singleton co-tree) and it is not hard to see that so is $(\T_2, \leq_p)$ (we prove in Proposition \ref{prop tau is iso} that this poset is order isomorphic to $(\omega, \leq) \times (\omega,\leq)$, a finite product of two WPOs and thus a WPO as well).
     
     We then observe that any co-tree $\X$ in $\T_{n+1}$ can be identified in a unique manner as a pair, consisting of: the upper part of $\X$, which is a co-tree in $\T_2$; and the lower part part of $\X$, which is a finite multiset of co-trees in $\T_n$.
     This is depicted in Figure \ref{Fig:structure} and formalized in (the discussion that accompanies) Lemma \ref{lem structure}. 
     In fact, it is shown in Lemma \ref{lem ord reflecting} that this identification induces an order reflecting map
     \[
     \pi \colon (\T_{n+1}, \leq_p) \to (\T_2, \leq_p) \times (\T_n^\#, <<),
     \]
     where $\T_n^\#$ denotes the set of finite multisets of $\T_n$ and $<<$ is
     the partial order introduced in Definition \ref{def multi proj}, which we termed the \textit{multiset projectivity relation}.
     
     If we could show that the statement ``if $(\T_n,\leq_p)$ is a WPO then so is $(\T_n^\#, <<)$'' holds true, we would be primed for an induction proof, because then the hypothesis that $(\T_n,\leq_p)$ is a WPO would yield that $(\T_2, \leq_p) \times (\T_n^\#, <<)$ is a finite product of two WPOs, hence it is itself a WPO, and thus the existence of the map $\pi$ would guarantee\footnote{This is because the image under $\pi$ of an infinite antichain in $\T_{n+1}$ would not be an infinite antichain in the WPO $(\T_2, \leq_p) \times (\T_n^\#, <<)$, so the order reflective property of this map yields a contradiction.
     Here, it might provide some intuition (and reassurance that our way of viewing the elements of $\T_{n+1}$ is almost faithful w.r.t. the order $\leq_p$) to note that the map $\pi$ is an \lq order embedding where it matters', since its restriction to $\T_{n+1} \smallsetminus \T_{n}$ is order invariant.
     And, if we are working under the assumption that $(\T_n, \leq_p)$ is a WPO, then to conclude that so is $(\T_{n+1}, \leq_p)$, it suffices to forbid infinite antichains in $\T_{n+1} \smallsetminus \T_{n}$, since any infinite antichain in $\T_{n+1}$ can only contain finitely many elements of $\T_n$.} that indeed $(\T_{n+1}, \leq_p)$ is a WPO.

     Unfortunately, the definition of a WPO appears to lack the strength to ensure that in general, the multiset projectivity relation on the finite multisets of a WPO must also be a WPO. 
     However, we prove in Theorem \ref{thm multi proj is bpo} that this desired `transfer property' holds when we restrict our setting to \textit{better partial orders} (BPOs for short), a particular type of WPOs discussed in Section \ref{Sec bpos}.
     
     Since Proposition \ref{prop tau is iso} shows that $(\T_2, \leq_p)$ is a product of two well orders, hence a BPO by Theorem \ref{thm bpo generation}, we can use the strategy detailed above for the inductive step, but with the term WPO replaced with BPO everywhere, to prove that every $(\T_{n}, \leq_p)$ is a BPO.
     Since BPOs are WPOs, and WPOs have no infinite antichains, this satisfies our initial goal.
\eroman

This paper is structured as follows.
In Section \ref{Sec bpos}, we present the necessary machinery to prove in Theorem \ref{thm multi proj is bpo} that our multiset projectivity relation defined on the set of finite multisets of a BPO must also be a BPO.
Section \ref{Sec bg} is just a brief overview of the duality between finite bi-Gödel algebras and finite co-forests, and of the properties of the variety $\bg$ we will need.
The beginning of Section \ref{Sec counting} explains why our main Theorem \ref{main thm} will follow from showing that every $(\T_n, \leq_p)$ has no infinite antichains, while the remainder of the section is dedicated to the proof that every $(\T_n, \leq_p)$ is a BPO.

For convenience and ease of reference, we compile in the following theorem all the aforementioned results concerning local tabularity in $\Lambda(\bigd)$.
We also present in Figure \ref{Fig:the lattice} a depiction of the dual of this lattice.

\begin{Theorem}\label{The big theorem}
Consider the lattice $\Lambda(\bigd)$ of axiomatic extensions of the bi-intuitionistic Gödel-Dummett logic 
\[
\vdash_{\lc} \coloneqq \cbipc + (p \to q) \lor (q \to p),
\]
i.e., the logic algebraized by the variety of bi-Gödel algebras $\bg$, and let $\cfc$ be the logic of the finite combs and $\vfc$ its variety.
    \benormal
        \item For $\vdash \in \Lambda(\bigd)$, the following conditions are equivalent:
            \benroman
                \item $\vdash \textit{ is locally tabular}$;
                \item ${\vdash} \nsubseteq \cfc$;
                \item $\vdash \J(\C_n) \textit{ for some } n \in \ZZ^+$;
                \item $\vdash \beta(\C_n) \textit{ for some } n \in \ZZ^+$.
            \eroman 

        \item For $\V \subseteq \bg$, the following conditions are equivalent:
            \benroman
                \item $\V$ is locally finite;
                \item $\vfc \nsubseteq \V$;
                \item $\V \textit{ omits the algebraic dual of a finite comb}$;
                \item there exists a natural bound for the size of the finite combs that can be order embedded into the bi-Esakia models of $\V$.
            \eroman 

        \item The logic $\cfc$ of the finite combs is decidable and is the unique pre-locally tabular logic in $\Lambda(\bigd)$.
        
        \item If $\vdash \in \Lambda(\bigd)$ is finitely axiomatizable, then determining if $\vdash$ is locally tabular is a decidable problem.

        \item If $\vdash \in \Lambda(\bigd)$ is not finitely axiomatizable, then it is not locally tabular.

        \item While $\Lambda(\bigd)$ has the size of the continuum, it contains only $\aleph_0$ locally tabular logics, all of which are finitely axiomatizable.
    \enormal 
\end{Theorem}

\begin{figure}[h]
\begin{tikzpicture}[scale=.5]
    \tikzstyle{point} = [shape=circle, thick, draw=black, fill=black , scale=0.35]
    \node [label=below:{$\mathbb{V}(\C_0')=\bacat=\mathbb{V}(\La_1)$}] at (0,0) [point] {};
    \node [label=left:{$\mathbb{V}(\C_1)$}] [label=right:{$\mathbb{V}(\La_2)$}] at (0,3) [point] {};
    \node [label=right:{$\mathbb{V}(\La_3)$}] at (1,4) [point] {};
    \node [label=right:{$\mathbb{V}(\La_4)$}] at (2,5) [point] {};
    \node [label=right:{$\V_{\bilc}$}] at (4,7) [point] {};
    \node [label=left:{$\mathbb{V}(\C_1')$}] at (-4,7) [point] {};
    \node [label=right:{$\V_2 $}] at (0,11) [point] {};
    \node [label=left:{$\mathbb{V}(\C_2)$}] at (-5,8) [point] {};
    \node [label=right:{$\V_2 + \mathbb{V} (\C_2)$}] at (-1,12) [point] {};
    \node [label=left:{$\mathbb{V}(\C_2')$}] at (-9,12) [point] {};
    \node [label=right:{$\V_3 $}] at (-5,16) [point] {};
    \node [label=right:{$\V_3 + \mathbb{V} (\C_3)$}] at (-6,17) [point] {};
    \node [label=left:{$\mathbb{V}(\C_3)$}] at (-10,13) [point] {};
    \node [label=below:{$\V_{FC}$}] (A) at (-14,17) [point] {};
    \node [label=above:{$\bg$}] (B) at (-14,25) [point] {};
    \node [label=above:{$\aleph_0$}] at (0,6) [] {};
    \node [label=above:{$\emptyset$}] at (-2.5,8.5) [] {};
    \node [label=above:{$\aleph_0$}] at (-5,11) [] {};
    \node [label=above:{$\emptyset$}] at (-7.5,13.5) [] {};
    \node [label=above:{$\aleph_0$}] at (-10,16) [] {};
    \node [label=above:{$2^{\aleph_0}$}] at (-14,20) [] {};

    \draw (0,0)--(0,3)--(2,5);
    \draw[dotted] (2,5)--(4,7);
    \draw (0,3)--(-10,13);
    \draw[dashed] (4,7)--(0,11)--(-4,7);
    \draw[dashed] (-5,8)--(-1,12);
    \draw[dashed] (-9,12)--(-5,16)--(-1,12);
    \draw[dashed] (-10,13)--(-6,17);
    \draw (0,11)--(-1,12);
    \draw (-5,16)--(-6,17);
    \draw[dotted] (-10,13)--(-14,17);
    \draw[dashed] (-6,17) to[bend right] (-14,25);
    \draw (-14,21) ellipse (2cm and 4cm);
\end{tikzpicture}
\caption{The lattice $\Lambda(\bg)$ of nontrivial subvarieties of bi-Gödel algebras.}
\label{Fig:the lattice}
\end{figure}
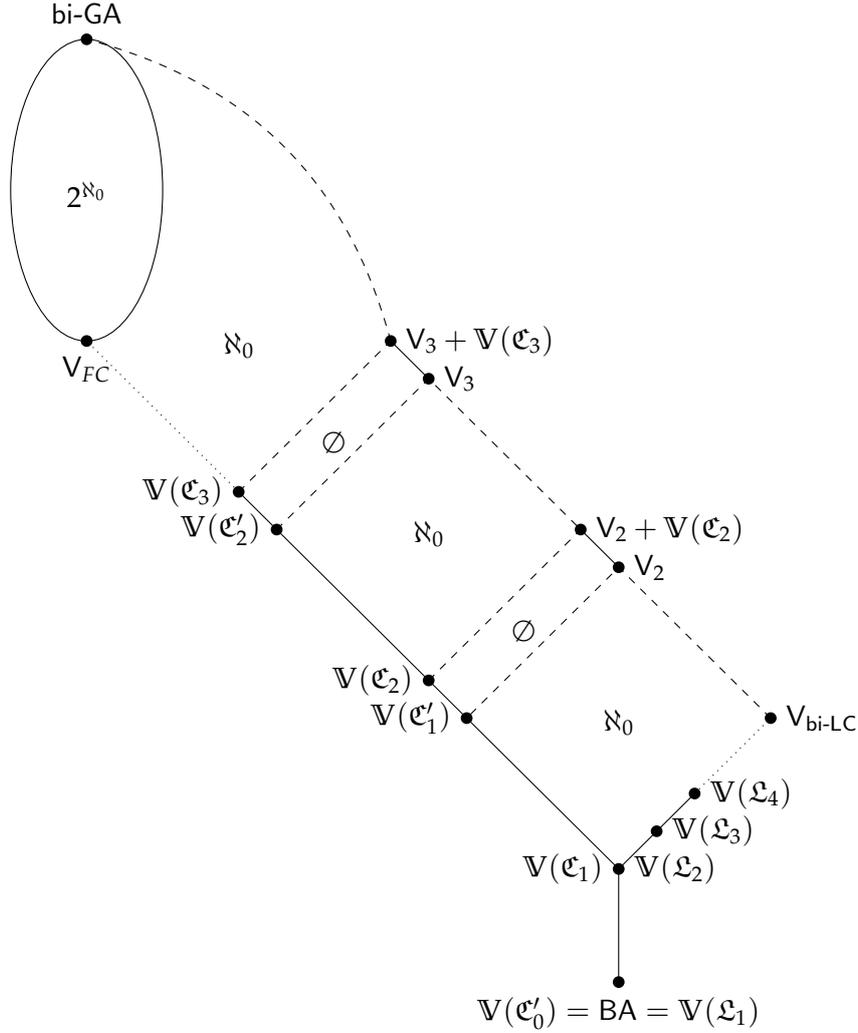

In Figure \ref{Fig:the lattice}: a full line connecting two points means that they are related by the immediate predecessor/successor relation; a dotted line represents a chain of order type $(\omega + 1)$; and a dashed line represents an infinite but unknown structure, which is not necessarily a chain.
We use $+$ for the join operation of the lattice $\Lambda(\bg)$.

Recall that $\vfc \coloneqq \HHH \SSS \PPP\{\C_n^* \colon n \in \ZZ^+\}$ is the variety which algebraizes the logic of the finite combs $\cfc$, and that for each positive $n$,
\[
\V_n \coloneqq \{\A \in \bg\colon \A \models \beta(\C_n)\} = \{\A \in \bg\colon \C_n \not \hookrightarrow \A_*\}.
\]
This lattice admits the following partition
\[
\Lambda(\bg) = \up \V_{FC} \biguplus \bigcup_{n\in \ZZ^+} \down \V_n,
\]
where $\up \V_{FC}$ contains all the $2^{\aleph_0}$ non-locally finite varieties of bi-Gödel algebras, while any of the $\aleph_0$ locally finite ones is a subvariety of some $\V_n$.
We note that since $\bg$ has the FMP, the only variety which contains all of the $\V_n$ is $\bg$ itself.

We denote by $\V_{\bilc}$ the variety of bi-Gödel algebras axiomatized by the bi-intuitionistic linear calculus (i.e., the bi-intuitionistic logics of chains)
\[
\vdash_{\operatorname{\mathsf{bi-LC}}} = \cbipc + (p \to q) \lor (q \to p) + \neg[ (q\gets p) \land (p \gets q)].
\]

For $n \in \omega$, we denote the $n$\textit{-chain} (i.e., a linearly ordered poset with $n$ elements) by $\La_n$, and the $n$\textit{-comb with handle} by $\C_n'$ (see Figure \ref{Fig:finite-hcombs2}).
\begin{figure}[h]
\centering
\begin{tabular}{c}
\begin{tikzpicture}
    \tikzstyle{point} = [shape=circle, thick, draw=black, fill=black , scale=0.35]
    \node [label=left:{$y_0$}] (0) at (0,0) [point] {};
    \node [label=right:{$y_1'$}] (1') at (1,0) [point] {};
    \node [label=left:{$y_1$}] (1) at (0.5,0.5) [point] {};
    \node [label=right:{$y_2'$}] (2') at (1.5,.5) [point] {};
    \node [label=left:{$y_2$}] (2) at (1,1) [point] {};
    \node [label=above:{$y_n$}] (n) at (1.75,1.75) [point] {};
    \node [label=right:{$y_n'$}] (n') at (2.25,1.25) [point] {};
    
    \draw (0)--(1)--(2);
    \draw (1')--(1);
    \draw (2')--(2);
    \draw (n')--(n);
    \draw [dotted] (2)--(n);
\end{tikzpicture}
\end{tabular}
\caption{The $n$-hcomb $\C_n'$.}
\label{Fig:finite-hcombs2}
\end{figure}
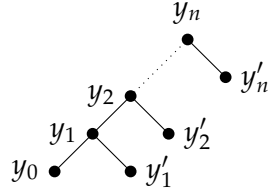
Given a finite co-tree $\X$, we write $\mathbb{V}(\X)$ as a shorthand for the variety $\HHH\SSS\PPP\{\X^*\}$ generated by the algebraic dual of $\X$.
We have that 
\[
\mathbb{V}(\La_1) \prec \mathbb{V}(\La_2) \prec \dots \prec \V_{\bilc} \hspace{.3cm} \text{and} \hspace{.3cm} \down \V_{\bilc} = \{\mathbb{V}(\La_1), \mathbb{V}(\La_2), \dots,\V_{\bilc}\},
\]
and 
\[
\mathbb{V}(\C'_0) \prec \mathbb{V}(\C_1) \prec \mathbb{V}(\C'_1) \prec  \mathbb{V}(\C_2) \prec \dots \prec \V_{FC} \hspace{.3cm} \text{and} \hspace{.3cm} \down \V_{FC} = \{\mathbb{V}(\C'_0), \mathbb{V}(\C_1), \dots,\V_{FC}\},
\]
where $\prec$ represents the immediate predecessor/successor relation.
A proper justification of this depiction can be found in Theorem \ref{thm depiction}.

\section{Better partial orders}\label{Sec bpos}
In this section we introduce the necessary concepts from order theory that we will need throughout.
We then use them to show that equipping the set of finite multisets of a BPO with the multiset projectivity relation yields a BPO.

When $A$ is a set and $n$ is an element of the natural numbers $\omega$, we denote the set of: finite subsets of $A$ by $[A]^{<\omega}$; infinite subsets of $A$ by $[A]^{\omega}$; subsets of $A$ of cardinality $n$ by $[A]^{n}$.

We will always identify a subset $B \subseteq \omega$ with its increasing enumeration. 
Given another $C \subseteq \omega$, we say that $C$ \textit{extends} $B$, in symbols $B \sqsubseteq C$, exactly when the increasing enumeration of $C$ extends that of $B$ (e.g., $\{2,4,6\} \sqsubseteq \{2,4,6,8,10\}$).
In other words, when $C$ is an initial segment of $B$ w.r.t. the natural order of $\omega$.

Let $(P, \leq_P)$ be \textit{poset}, i.e., a set $P$ equipped with a \textit{partial order} $\leq_P$ (a binary relation on $P$ that is reflexive, transitive, and antisymmetric).
Sometimes we simply write $\leq$ or $P$ instead of $\leq_P$ or $(P, \leq)$, respectively.
If $p,q\in P$ are such that $p \nleq q$ and $q \nleq p$, they are said to be \textit{incomparable}.
Otherwise, they are \textit{comparable}.
If $P$ has no incomparable elements, we call $\leq$ a \textit{linear order}.

A \textit{subposet} of $P$ is a subset $Q$ equipped with the induced order $\leq_Q \coloneqq Q^2 \cap \leq_P$, so we can always view subsets as subposets and write $(Q, \leq_Q)$ or $(Q,\leq)$ or $Q$ interchangeably.
If $Q$ is a subposet of $P$ such that $\leq_Q$ is a linear order, then $Q$ is said to be a \textit{chain} (in $P$).

When a subposet $Q$ has a least element (i.e., a point $t \in Q$ satisfying $t \leq q$ for every $q \in Q$), we call it the \textit{minimum} of $Q$ and denote it by $Min(Q)$.
A point $t \in Q$ is \textit{minimal} (in $Q$) if $q\leq t$ implies $q =t$ for all $q\in Q$.
The definitions of the \textit{maximum} of a subposet and of its \textit{maximal} points are analogous, hence omitted.

Given $A \subseteq P$, the set $\down A \coloneqq \{p\in P \colon \exists a \in A \, (p\leq a)\}$ is called the \textit{downset (of $P$) generated} by $A$ and if $A = \down A$, then $A$ is a \textit{downset} (of $P$).
When $A=\{a\}$ we simply write $\down a \coloneqq \down A$ and call it a \textit{principal downset}.
The arrow operator $\up$ and the notion of the (principal) upsets of $P$ are defined analogously, hence omitted.
We denote the set of downsets of $P$ by $Down(P)$ and that of its upsets by $Up(P)$.

Let $Q$ be another poset.
An \textit{order embedding} of $P$ into $Q$ is a map $h \colon P \hookrightarrow Q$ that is \textit{order invariant}, i.e., 
\[
p \leq_P p' \iff h(p) \leq_Q h(p')
\]
holds for all $p,p' \in P$.
The left to right implication of the previous equivalence is called \textit{order preservation}, while its converse is called \textit{order reflection}.
It is easy to see that order embeddings are necessarily injective.
If moreover the map $h$ is surjective, we call it an \textit{order isomorphism}, and say that the posets $P$ and $Q$ are \textit{(order) isomorphic}.

A straightforward argument ensures that $P$ order embeds into $Q$ iff $P$ is isomorphic to a subposet of $Q$.
We sometimes say that in this case, there exists a copy of $P$ inside of $Q$, and denote this by $P \hookrightarrow Q$.
Otherwise, we write $P \not \hookrightarrow Q$.

The \textit{(direct) product} of these posets is denoted $(P \times Q, \leq_P \times \leq_Q) \coloneqq (P, \leq_P) \times (Q, \leq_Q)$, where, for all $p,p'\in P$ and $q,q' \in Q$, the order is defined by
\[
(p,q) \leq_P \times \leq_Q (p',q') \iff p \leq_P p' \text{ and } q \leq_Q q'.
\]

An \textit{(infinite) sequence} in $P$ is a map $f \colon \omega \to P$.
It is often convenient to denote a sequence $f$ by $\seq$, where $p_i \coloneqq f(i)$ for each $i \in \omega$.
We call a sequence $\seq$ in $P$ \textit{bad} when $p_i \nleq p_j$ for all $i < j \in \omega$.
Two particular types of bad sequences in $P$ are:
\benroman
    \item \textit{infinite descending chains}, i.e., sequences $\seq$ such that $p_i > p_{i+1}$ for all $i \in \omega$;

    \item \textit{infinite antichains}, i.e., sequences $\seq$ whose elements are pairwise incomparable.
\eroman

We say that $P$ is: \textit{well-founded} if it has no infinite descending chains; a \textit{well partial order} (WPO for short) if it is well-founded and has no infinite antichains; a \textit{well-order} if it is well-founded and $\leq$ is linear.

\begin{Proposition}\label{prop wpos}
    The following conditions hold for a poset $P$:
    \benroman
    \item $P$ is well-founded iff every nonempty linearly ordered subposet has a minimum;
    \item\cite[Prop.~I.1.1]{BPO} $P$ is a WPO iff there is no bad sequence in $P$ iff $(Down(P), \subseteq)$ is well-founded. 
    \eroman 
\end{Proposition}

\noindent \textit{Proof sketch.}
A common proof strategy for the first equivalence of condition (ii) is to extract an infinite descending chain or an infinite antichain from an arbitrary bad sequence in P.
This can be easily done using Ramsey's Theorem, a result whose generalization presented in Proposition \ref{prop front ramsey} is of the utmost importance to the theory of BPOs.
Given this, seeing a simple application of a `Ramsey argument' might provide some intuition for what follows.

For suppose that we have a bad sequence $\seq$ in P.
Then for any $i < j \in \omega$ we have $p_i \nleq p_j$, which implies either $p_j < p_i$ or that $p_i$ and $p_j$ are incomparable.
In case of the former we set $g(\{i,j\}) \coloneqq 0$, otherwise we set $g(\{i,j\}) \coloneqq 1$. 
This defines a map $g\colon [\omega]^2 \to 2$ and Ramsey's Theorem ensures the existence of an infinite subset $H\subseteq \omega$ such that the image of $g$ is constant when restricted to $[H]^2$.
We can now construct in $P$ an infinite descending chain if the aforementioned constant image is $0$, or an infinite antichain if the constant image is $1$, thus proving that $P$ is not a WPO. \qed

\vspace{.3cm}

Let $P$ be a WPO.
In view of the second equivalence of Proposition \ref{prop wpos}.(ii), it is natural to ask what conditions one needs to impose on $P$ in order to forbid infinite antichains in $(Down(P), \subseteq)$, thus yielding another WPO.
That additional properties are indeed required for this transfer of WPOness can be witnessed by Rado's Poset (see Figure \ref{fig:Rado}), a WPO with infinite antichains of downsets.
This example is minimal in the sense that if $Q$ is a WPO but $(Down(Q), \subseteq)$ is not, then Rado's poset order embeds into $Q$ (see, e.g., \cite[Ex.~I.1.2]{BPO} and the subsequent discussion).

It turns out that not allowing \textit{bad sequences of sequences} in $P$, i.e., maps $f\colon [\omega]^2 \to P$ satisfying, for $m,n,l \in \omega$,
\[
m<n<l \implies f(\{m,n\}) \nleq f(\{n,l\}),
\]
is enough to ensure that $(Down(P), \subseteq)$ is a WPO.
And if one wants to guarantee that \\
$(Down(Down(P)), \subseteq)$ will also be a WPO, then it suffices to forbid `bad sequences of sequences of sequences' in $P$, and so on.
This idea of `bad sequences of .... of sequences' can be formalized using fronts and super-sequences, and forbidding bad super-sequences leads to the notion of better partial orders.
We proceed to properly define all this terminology.

\begin{figure}
    \centering
    \includegraphics[width=0.5\linewidth]{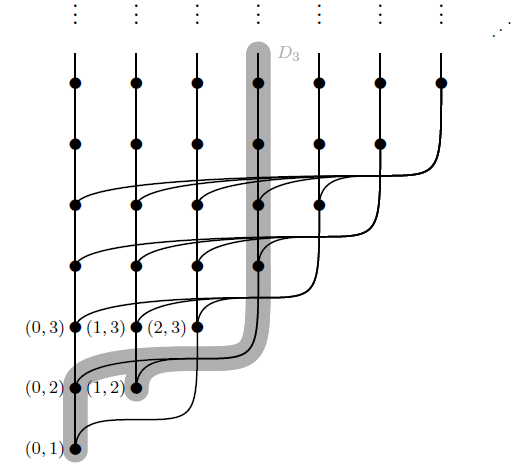}
    \caption{Rado's Poset}
    \label{fig:Rado}
\end{figure}

\begin{defi}\label{def fronts}
    Let $A \in [\omega]^\omega$ and $F \subseteq [\omega]^{< \omega}$.
    We say that $F$ is a \textit{front} on $A$ when the following conditions hold:
    \benroman
        \item either $F=\{\emptyset\}$ or $\bigcup F = A$;
        
        \item $\forall s,t \in F \, \big( s \sqsubseteq t \implies s=t \big);$

        \item $\forall B \in [A]^\omega, \exists! s \in F \, \big(s \sqsubseteq B).$
    \eroman
    In the case that $F=\{\emptyset\}$, we call it the \textit{trivial front} and note that by definition, it is a front on every $B \in [\omega]^\omega$.
    On the other hand, when $F$ is a nontrivial front on $A$, then $A$ is unique for this property since $\bigcup F=A$.
    Because of this, we can always introduce such a nontrivial front $F$ without explicit mention of $A$.
\end{defi}

\begin{exa}
    The set $[\omega]^n$ of subsets of $\omega$ with cardinality $n\in \omega$ is always a front on $\omega$.
\end{exa}

Given $s,t \in [\omega]^{< \omega}$ we say that $t$ is a \textit{shift} of $s$, in symbols $s \lhd t$, when there exists $B \in [\omega]^\omega$ such that $s \sqsubseteq B$ and $t \sqsubseteq B \smallsetminus Min(B)$.

\begin{defi}
    Let $\leq$ be a partial order on a set $P$ and let $F$ be a front.
    \benroman
        \item A map $f\colon F \to P$ is called a \textit{super-sequence} in $P$.

        \item A super-sequence $f\colon F \to P$  is said to be \textit{bad for $\leq$} if $s \lhd t$ implies $f(s) \nleq f(t)$, for all $s,t \in F$.

        \item A \textit{sub-super-sequence} of a super-sequence $f\colon F \to P$ is a restriction $f {\restriction_G} \colon G \to P$ of $f$ to some front $G \subseteq F$.
    \eroman
    A poset $(P,\leq)$ is a \textit{better partial order} (BPO for short) if there are no bad super-sequences for $\leq$.
\end{defi}

Recall that WPOs admit the definition of being posets without bad sequences.
Noting that a sequence $f \colon \omega \to P$ in a poset $P$ can always be regarded as a super-sequence $f \colon [\omega]^1 \to P$, it is clear that BPOs are a particular case of WPOs.
Crucially for our purposes, BPOs satisfy the following Ramsey-like property, which follows as a corollary to the Nash-Williams Theorem for fronts (see, e.g., \cite[Thm.~I.2.11]{BPO}).

\begin{Proposition}[{\cite[Prop.~I.3.13]{BPO}}] \label{prop front ramsey}
    Let $R$ a binary relation on a set $P$ and let $F$ be a front.
    Every super-sequence $f \colon F \to P$ admits a sub-super-sequence $f {\restriction_G} \colon G \to P$ that satisfies exactly one of the following conditions:
    \benroman
        \item $\forall s,t \in G \, \big( s \lhd t \implies f{\restriction_G}(s) R f{\restriction_G}(t) \big)$;

        \item $\forall s,t \in G \, \big( s \lhd t \implies \neg( f{\restriction_G}(s) R f{\restriction_G}(t))\big)$.
    \eroman
\end{Proposition}

Having defined BPOs, we can present some notable properties.

\begin{Theorem}\label{thm bpo generation}
    Let $P$ and $Q$ be posets.
    \benroman
        \item If $P$ is a well-order then it is a BPO.

        \item\cite[Lem.~I.310]{BPO} If $Q$ is a BPO and there exists an order reflecting map $h\colon P \to Q$, then $P$ is also a BPO.

        \item\cite[Prop.~I.3.14]{BPO} If both $P$ and $Q$ are BPOs then so is their product $(P \times Q, \leq_P \times \leq_Q)$.

        \item if $P$ is a BPO then so is $(Down(P), \subseteq)$.
    \eroman 
\end{Theorem}

\noindent\textit{Proof Sketch.}
    Since all the above conditions are used to show that the partial orders investigated in this paper are in fact BPOs, we hope to provide a bit of intuition with some proof sketches and remarks.
    \benroman
        \item Suppose we have a super-sequence $f \colon F \to P$ which is bad for a well-order $\leq$ on $P$.
        Let $A \coloneqq \bigcup F$ and take $s_0,s_1\in F$ satisfying $s_0 \lhd s_1$, i.e., $s_0 \sqsubseteq B$ and $s_1 \sqsubseteq B \smallsetminus Min(B)$ for some $B \in [\omega]^\omega$.
        We set $B_1 \coloneqq \{a \in A \colon Min(s_1)\leq a\} \in [A]^\omega$, noting that $s_1 \sqsubseteq B_1$.
        It is clear that $B_1 \smallsetminus Min(B_1) \in [A]^\omega$, so using condition (iii) of Definition \ref{def fronts} yields a unique $s_2 \in F$ such that $s_2 \sqsubseteq B_1 \smallsetminus Min(B_1)$, hence $s_1 \lhd s_2$ follows.
        This procedure constructs an infinite sequence $(s_i)_{i \in \omega}$ satisfying $s_i \lhd s_{i+1}$ for all $i \in \omega$.
        But we assumed $f$ to be bad for $\leq$, so $s_i \lhd s_{i+1}$ forces $f(s_i) \nleq f(s_{i+1})$, and since $\leq$ is a well-order, hence linear, we obtain $f(s_{i+1}) < f(s_i)$.
        Thus, $(f(s_i))_{i \in \omega}$ is an infinite descending chain in the well-order $(P, \leq)$, a contradiction.

        \item If a super-sequence $f \colon F \to P$ is bad for $\leq_P$, then any order reflecting map $h\colon P \to Q$ yields a super-sequence $h \circ f \colon F \to Q$ which is bad for $\leq_Q$.

        \item Suppose that $f \colon F \to P \times Q$ is bad for $\leq_P \times \leq_Q$ and define a binary relation $R$ on $P\times Q$ by: $(p,q) R (p',q')$ iff $p \leq_P p'$.
        Then Proposition \ref{prop front ramsey} ensures the existence of a sub-super-sequence of $f$ that can either be used to construct a super-sequence which is bad for $\leq_Q$, or one which is bad for $\leq_P$.

        \item As previously mentioned, this property is one of the features of BPOs that motivated their investigation, and is a particular case of a phenomenon called \textit{infinite stability}.
        Broadly, it states that if $P \mapsto \mathcal{O}(P)$ is an infinitary operation (in the sense that elements of $\mathcal{O}(P)$ can be viewed as objects labeled by (possibly infinitely many) elements of $P$), then $\mathcal{O}(P)$ is a BPO if so is $P$.
        This is proved, e.g. in \cite[Prop.~I.3.18]{BPO}, and the current condition (iv) follows as an immediate corollary.
    \eroman

When $A$ is a set, a \textit{finite multiset} of $A$ is a finite list $[a,b,c,\dots]$ of possibly repeated elements of $A$ in which the order does not matter.
In other words, a finite collection of elements of $A$ where the same element can appear multiple times.
Formally, we define them as maps $M \colon B \to \omega \smallsetminus \{0\}$ where $B \in [A]^{< \omega}$ is called the \textit{universe} of $M$ and for $b \in B$, we call $M(b)$ the \textit{multiplicity} of $b$ (which represents the number of occurrences of $b$ in the multiset $M$).

Denote the set of finite multisets of $A$ by $A^\#$ and let $M \colon B \to \omega \smallsetminus \{0\}$ be one such finite multiset. 
When $B=\{b_1, \dots , b_n\}$ we identify the multiset $M$ with every list $[a,b,c, \dots]$ that only contains elements from $B$ and each $b_i\in B$ occurs exactly $M(b_i)$-times.
When we refer to an \textit{element} $b$ of $M$, in symbols $b \in M$, we are not only referring to the element $b \in B$, but to a specific occurrence of $b$ in $M$.
We define the \textit{length} of $M$ as $l(M)\coloneqq \sum_{b\in B} M(b)$.

A \textit{sub-multiset} of $M \colon B \to \omega \smallsetminus \{0\}$ is any finite multiset $M' \colon B' \to \omega \smallsetminus \{0\}$ of $A$ satisfying $B'\subseteq B$ and $M'(b)\leq M(b)$ for all $b \in B'$.
It is easy to see that in this case, $l(M') \leq l(M)$ always holds.

If $N\colon C \to \omega \smallsetminus \{0\}$ is another finite multiset of $A$, then a \textit{map} $f\colon M \to N$ is an assignment of each element of $M$ to exactly one of $N$.
We denote the \textit{image} of such $f$ by $Im(f)\coloneqq [f(a) \colon a\in M]$, noting that it is clearly a sub-multiset of $N$.
We say that $f$ is:
\benroman
    \item \textit{injective} when $l(M)=l(Im(f))=l([f(a) \colon a\in M])$;

    \item \textit{surjective} when $l(Im(f))=l([f(a) \colon a\in M])=l(N)$, i.e., when $Im(f)=N$.
\eroman


\begin{defi}
    Let $(P,\leq)$ be a poset. We define the \textit{multiset embeddability} relation $\preceq$ on the set $P^\#$ of finite multisets of $P$ by: $N \preceq M$ iff there exists an injective map $f \colon N \hookrightarrow M$ such that $p \leq f(p)$ for all $p \in N$.
\end{defi}

Notice that, if $f \colon N \hookrightarrow M$ is an injective map between finite multisets of a poset $P$, then it might be the case that for distinct $p,q \in N$, the images $f(p)$ and $f(q)$ represent the same element $r$ in $P$, but in $M$ they must represent different occurrences of $r$.   
For example, in $(\omega^\#, \preceq)$ we have $[2,5,2] \preceq [6,3,3,1]$ while $[2,5,2] \npreceq [6,3,1]$.

\begin{Lemma}[Higman's Lemma {\cite[Thm.~4]{Higman}}]\label{prop higman}
    If $P$ is a BPO then so is $(P^\#, \preceq)$.
\end{Lemma}

In order to prove the main result of this paper, we will need a stronger version of the multiset embeddability relation:

\begin{defi}\label{def multi proj}
    Let $(P,\leq)$ be a poset. We define the \textit{multiset projectivity relation} relation $<<$ on the set $P^\#$ of finite multisets of $P$ by: $N << M$ iff there exists a surjective map $f \colon M \twoheadrightarrow N$ such that $ f(p) \leq p$ for all $p \in M$.
\end{defi}

We first show that the multiset projectivity relation is a partial order, and subsequently that it is indeed stronger than multiset embeddability relation.

\begin{Lemma}
    If $P$ is a poset then so is $(P^\#, <<)$.
\end{Lemma}
\begin{proof}
    Given a poset $P$, the $<<$ relation on $P^\#$ is obviously reflexive and transitive.
Let us see why it is also antisymmetric, hence a partial order on $P^\#$.
Let $N,M\in P^\#$ and suppose $N << M <<N$, so there are $f \colon M \twoheadrightarrow N$ such that $ f(p) \leq p$ for all $p \in M$, and $g \colon N \twoheadrightarrow M$ such that $ g(q) \leq q$ for all $q \in N$.

We proceed by induction on the length of $M$, noting that since these are maps between multisets we always have
\[
l(Im(f))\leq l(M) \hspace{.3cm} \text{and }\hspace{.3cm} l(Im(g)) \leq l(N),
\]
while their surjectivity ensures 
\[
l(Im(f))=l(N) \hspace{.3cm} \text{and }\hspace{.3cm} l(Im(g))=l(M),
\]
and thus
\[
l(M)=l(Im(f))=l(Im(g))=l(N).
\]
It follows that both maps are also injective.

The case $l(M)=1$ is an immediate consequence of $\leq$ being anti-symmetric.
We now assume that pairwise $<<$-related multisets of length $n-1$ must be equal and that $l(M)=n$.
Let $U_M$ and $U_N$ be the respective universes of $M$ and $N$.
Since $U_M$ is a finite subset of $P$ by definition, it must contain a point $p$ which is minimal in $U_M$.
Let $p'$ be one of the occurrences of $p$ in $M$.
By the assumptions on $f$ and $g$ we know that $g\circ f(p') \leq f(p') \leq p'$.
But $g \circ f(p') \in U_M$, so the minimality of $p$ in $U_M$ forces $g\circ f(p')$ to be an occurrence of $p$ in $M$.
Using the antisymmetry of $\leq$ we can now infer that $f(p')$ must be an occurrence of $p$ in $N$.
This shows that the image under $f$ of any occurrence of $p$ in $M$ is necessarily mapped to an occurrence of $p$ in $N$.
By a similar argument we can also show that the image under $g$ of any occurrence of $p$ in $N$ is necessarily mapped to an occurrence of $p$ in $M$.

Since $l(M)=l(N)$ and our maps are injective, we can delete exactly one such occurrence of $p$ in $M$ and $N$ and obtain sub-multisets $M'$ and $N'$ of $M$ and $N$ respectively, that satisfy $l(M')=l(N')=n-1$ and $M' << N' << M'$.
Using the induction hypothesis we get $M'=N'$, and by simply adding one occurrence of $p$ to each of these multisets, it is now clear that $M=N$.
\end{proof}

\begin{Lemma}\label{lem stronger rel}
    Let $P$ be a poset and $M,N \in P^\#$. We have $N << M$ iff $N \preceq M$ and for every $p\in M$, there exists $q \in N$ such that $q \leq p$.
\end{Lemma}
\begin{proof}
    To prove the left to right implication, suppose that $N << M$, i.e., that there exists a surjective map $f \colon M \twoheadrightarrow N$ such that $f(p) \leq p$ for all $p \in M$.
    Because of this last condition, we only need to prove $N \preceq M$.
    Since $f$ is surjective, for $q \in N$ there exists $p \in M$ satisfying $q=f(p) \leq p$.
    We fix one such $q_p \in M$ and set $g(q) \coloneqq q_p$, noting that $q=f(q_p)\leq q_p =g(q)$.
    This clearly defines an injective map $g\colon N \hookrightarrow M$ that makes $N \preceq M$ hold true.

    Conversely, assume that $N \preceq M$ and for every $p\in M$, there exists $q \in N$ such that $q \leq p$.
    From $N \preceq M$ we obtain an injective map $g\colon N \hookrightarrow M$ satisfying $q \leq g(q)$ for all $q \in N$.
    We now define a map $f \colon M \to N$. 
    If $p\in M$ but $p \notin Im(g)$, we take some $q \in N$ such that $q \leq p$, which exists by our above assumption, and set $f(p)\coloneqq q$, noting that $f(p) = q \leq p$.
    If $p \in Im(g)$, then by the assumptions on $g$ there exists a unique $q\in N$ such that $q \leq g(q)=p$.
    We then set $f(p) \coloneqq q$, noting $f(p)=q\leq g(q)=p$ and that because $g$ is injective, this definition ensures that the map $f$ will be surjective.
    Thus, we do have $N<<M$.
\end{proof}

Finally, we can accomplish the goal of this section:

\begin{Theorem}\label{thm multi proj is bpo}
    If $P$ is a BPO then so is $(P^\#, <<)$.
\end{Theorem}
\begin{proof}
    Let $(P, \leq)$ be a BPO and assume, with a view to contradiction, that there exists a super-sequence $f'\colon F \to P^\#$ which is bad for $<<$.
    Consider the multiset embeddability relation $\preceq$ on $P^\#$.
    Since it is a binary relation on $P^\#$, Proposition \ref{prop front ramsey} forces the existence of a sub-super-sequence $f'{\restriction_G}\colon G \to P^\#$ that satisfies exactly one of the following conditions:
    \benroman
        \item $\forall s,t \in G \, \big( s \lhd t \implies f'{\restriction_G}(s) \preceq f'{\restriction_G}(t) \big)$;

        \item $\forall s,t \in G \, \big( s \lhd t \implies  f'{\restriction_G}(s) \npreceq f'{\restriction_G}(t)\big)$.
    \eroman
    Notice that condition (ii) is equivalent to $f'{\restriction_G}\colon G \to P^\#$ being a bad super-sequence for $\preceq$, which implies that $(P^\#, \preceq)$ is not a BPO and thus contradicts Higman's Lemma \ref{prop higman}.
    Hence, we know that condition (i) must hold.
    Combining this with the assumption that $f'$ is bad for $<<$ and setting $f\coloneqq f'{\restriction_G}\colon G \to P^\#$ yields
    \[
    \forall s,t \in G \, \big( s \lhd t \implies f(s) \preceq f(t) \text{ and } \neg (f(s) < < f(t)) \big).
    \]
    Using the equivalence stated in Lemma \ref{lem stronger rel}, we can rewrite the above display as
    \[
    \forall s,t \in G \, \big( s \lhd t \implies f(s) \preceq f(t) \text{ and } \exists p\in f(t), \forall q \in f(s) \, (q \nleq p)  \big),
    \]
    and for each pair of elements $s,t\in G$ satisfying $s \lhd t$, we fix a witness $p_s^t\in f(t)$ such that $q \nleq p_s^t$ for all $q \in f(s)$.

    We finish this proof by showing that setting $h(s) \coloneqq \down \{p_s^t \colon t\in G \text{ and } s \lhd t\}$ for each $s \in G$ defines a super-sequence $h \colon G \to Down(P)$ which is bad for $\subseteq$.
    This will provide us with the desired contradiction because we assumed that $(P,\leq)$ is a BPO, hence Theorem \ref{thm bpo generation}.(iv) ensures that so is $(Down(P),\subseteq)$.
    
    Accordingly, we take $s,t \in G$ such that $s \lhd t$ and prove $h(s) \nsubseteq h(t)$, thus showing that $h$ is bad for $\subseteq$.
    Since $h(s) = \down \{p_s^{t'} \colon t'\in G \text{ and } s \lhd t'\}$ and $s \lhd t$, it is clear that $p_s^t \in h(s)$, so it suffices to show $p_s^t \notin h(t) = \down \{p_t^{r'} \colon r'\in G \text{ and } t \lhd r'\}$.
    We take $r \in G$ satisfying $t \lhd r$ and recall the definition of $p_t^r$ as a point such that $p_t^r \in f(r)$ and $q \nleq p_t^r$ for all $q \in f(t)$.
    Since by the same definition we have $p_s^t\in f(t)$, it now follows that $p_s^t \nleq p_t^r$, thus $p_s^t \notin  \down \{p_t^{r'} \colon r'\in G \text{ and } t \lhd r'\} = h(t)$, as desired.    
\end{proof}

\section{Bi-Gödel algebras and co-trees}\label{Sec bg}


We fix a denumerable set $Prop$ of variables (usually denoted by $p,q,r \dots$).
When $\lang$ is an algebraic language, we denote the set of formulas of $\lang$ built up from $Prop$ by $Fm$ and the corresponding algebra of formulas by $\form$, whose endomorphisms will be called \textit{substitutions}.

A \textit{logic} $\vdash$ (in the language $\lang$) is a \textit{finitary} consequence relation on $Fm$ that is \textit{substitution invariant}, i.e., if $\sigma$ is a substitution and $\Gamma \cup \{\phi\}\subseteq Fm$, then $\Gamma \vdash \phi$ implies $\sigma[\Gamma_0] \vdash \sigma(\phi)$ for some finite $\Gamma_0 \subseteq \Gamma$.

A logic $\vdash$ is said to be \textit{consistent} if $\emptyset \nvdash \bot$. 
Given another logic $\vdash^+$, we call it an \textit{extension} of $\vdash$ when ${\vdash} \subseteq {\vdash^+}$.
If moreover there exists a set of formulas $\Sigma$ closed under substitutions and such that
    \[
    \Gamma \vdash^+ \phi \iff \Gamma \cup \Sigma \vdash \phi
    \]
for all $\Gamma \cup \{\phi\}\subseteq Fm$, we say that $\vdash^+$ is an \textit{axiomatic extension} of $\vdash$ and write ${\vdash^+} = {\vdash} + \Sigma$, or simply ${\vdash} + \psi$ when $\Sigma = \{\sigma(\psi) \colon \sigma \text{ is a substitution}\}$.
We will often use the following standard abbreviations: 
    \begin{itemize}
        \item $\vdash \phi$ when $\emptyset \vdash \phi$;

        \item $\Gamma, \Sigma \vdash \phi$ when $\Gamma \cup \Sigma \vdash \phi$;

        \item and $\Gamma \vdash \Sigma$ when $\Gamma \vdash \phi$ holds for all $\phi \in \Sigma$.
    \end{itemize}

We denote by $\mathbb{H}, \mathbb{S}$, and $\mathbb{P}$ the class operators of closure under homomorphic images, subalgebras, and (direct) products, respectively. 
A variety $\V$ is a class of (similar) algebras closed under $\mathbb{H}, \mathbb{S}$, and $\mathbb{P}$. By Birkhoff's Theorem, varieties coincide with classes of algebras that can be axiomatized by sets of equations (see, e.g., \cite[Thm.\ II.11.9]{Sanka2}). The smallest variety $\mathbb{V}(\mathsf{K})$ containing a class $\mathsf{K}$ of algebras is called the \textit{variety generated by $\mathsf{K}$} and coincides with $\HHH\SSS\PPP(\mathsf{K})$ (see, e.g., \cite[Thm.\ II.9.5]{Sanka2}).

Given an algebra $\A$, we denote its congruence lattice by $Con(\A)$.\ We say that $\A$ is \textit{subdirectly irreducible} (SI for short) if $Con(\A)$ has a second least element, and that $\A$ is \textit{simple} if $Con(\A)$ has exactly two elements.

When $\mathsf{K}$ is class of algebras, we denote by $\mathsf{K}^{<\omega}$, $\mathsf{K}_{SI}$, and $\mathsf{K}_{SI}^{<\omega}$ the classes of finite members of $\mathsf{K}$, SI members of $\mathsf{K}$, and SI members of $\mathsf{K}$ which are finite, respectively. It is a consequence of the Subdirect Decomposition Theorem (see, e.g., \cite[Thm.\ II.8.6]{Sanka2}) that if $\mathsf{K}$ is a variety, then $\mathsf{K} = \mathbb{V}(\mathsf{K}_{SI})$.


A \textit{bi-Heyting algebra} is an algebra $\A = (A, \land , \lor,\to,  \gets, 0,1)$ whose $(\land,\lor,0,1)$-reduct is a bounded distributive lattice and such that the \textit{implication} $\to$ and the \textit{co-implication} $\gets$ operations satisfy the following \textit{residuation laws}: for all $a,b,c \in A$, we have 
    \[
    c\leq a \to b \iff c \land a \leq b \hspace{.3cm} \text{ and } \hspace{.3cm} a \gets b \leq c \iff a \leq b \lor c.
    \]
We will denote the class of bi-Heyting algebras by $\bivar$.
Notably, $\bivar$ can be equationally defined (see, e.g., \cite{Rauszer3}) as the subclass of bounded distributive lattices that validate the equations 
    \begin{multicols}{2}
        \begin{enumerate}
            \item $p \to p \approx 1,$
            \item $p \land (p \to q) \approx p \land q$,
            \item $q\land ( p \to q) \approx q$,
            \item $p \to (q \land r) \approx (p\to q) \land (p \to r)$,
            \item $p \gets p \approx 0$,
            \item $p \lor (q \gets p) \approx p \lor q$,
            \item $q \lor (q \gets p) \approx q$,
            \item $(q\lor r)\gets p \approx (q \gets p) \lor (r \gets p)$.
        \end{enumerate}
    \end{multicols} 
\noindent Using the well-known fact that the class of bounded distributive lattices also admits an equational axiomatization, it is then immediate from Birkhoff's Theorem that $\bivar$ is a variety.

When $\phi$ is a formula (in the language of bi-Heyting algebras) and $\{\A\} \cup \mathsf{K}  \subseteq \bivar$, we write $\A \models \phi$ instead of $\A \models \phi \approx 1$, and write $\mathsf{K} \models \phi$ when $\B \models \phi$ for all $\B \in \mathsf{K}$.

The \textit{bi-intuitionistic propositional calculus} $\cbipc$ is the logic algebraized (in the sense of \cite{BP89}) by the variety $\bivar$ and by the sets $\tau(p)=\{p \approx 1\}$ and $\Delta(p,q)=\{p \leftrightarrow q\}$.
Consistent axiomatic extensions of $\cbipc$ are termed \textit{bi-intermediate logics} and the greatest (w.r.t. inclusion) such logic coincides with the \textit{classical propositional calculus} 
\[\ccpc \coloneqq {\cbipc} + p \lor \neg p,\]
where the co-implication is term-definable since $\ccpc (p \gets q) \leftrightarrow (p \land \neg q)$.
    
Notably, bi-intermediate logics are exactly the logics algebraized by the nontrivial subvarieties $\V \subseteq \bivar$ and by the same sets $\tau$ and $\Delta$.
Given such a variety $\V$, we will denote its bi-intermediate logic by $\vdash_\V$ and conversely, when $\vdash$ is a bi-intermediate logic we denote its variety of bi-Heyting algebras by $\V_\vdash$.
It follows from \cite{BP89} that there is a dual isomorphism between $\Lambda(\vdash)$, the lattice of consistent axiomatic extensions of $\vdash$, and $\Lambda(\V)$, the lattice of nontrivial subvarieties of $\V$.

In this paper, we will only be concerned with a particular class of bi-intermediate logics (and their varieties of bi-Heyting algebras), namely, those which extend the \textit{bi-Gödel-Dummett logic}
\[
\vdash_{\lc} \coloneqq \cbipc + (p \to q) \lor (q \to p).
\]
This logic is algebraized by the variety of \textit{bi-Gödel algebras}
\[
\bg \coloneqq \V_{\vdash_{\lc}} = \{\A \in \bivar \colon \A \models (p \to q) \lor (q \to p) \}.
\]
We list some properties of this variety that will be useful in what follows.

\begin{Theorem}\label{props of bg}
The following conditions hold true:
    \benroman
        \item \cite[Cor.~3.8]{Paper1} The variety $\bg$ is \textit{semi-simple}, i.e., every SI member of $\bg$ is simple.
        \item \cite[Prop.~2.2 \& Cor.~3.8]{Paper1} The variety $\bg$ is \textit{congruence distributive}, i.e., every member of $\bg$ has a distributive congruence lattice.
        \item \cite[Cor.~3.9]{Paper1} Every subalgebra of an SI bi-Gödel algebra is also SI.

        \item\cite[Thm.~4.16]{Paper1} The lattice $\Lambda(\bg)$ of nontrivial subvarieties of bi-Gödel algebras has cardinality $2^{\aleph_0}$.
    \eroman 
\end{Theorem}

Using the bi-Esakia duality, a restricted version of the celebrated Esakia duality \cite{Esakia2}, it can be shown that $\cbigd$ is the bi-intuitionistic logic of \textit{co-trees}, i.e., posets with a greatest element (called the \textit{co-root}) and whose principal upsets are chains.
Since we will only have to handle finite co-trees, we only present here the simpler finite version of the aforementioned duality for bi-Heyting algebras.

The bi-Esakia dual of $\A \in \bivar^{< \omega}$ is the (finite) poset $\A_*\coloneqq (A_*,\subseteq)$ of the prime filters of $\A$ ordered by inclusion.
Conversely, the bi-Heyting dual of a finite poset $\X$ is the (finite) algebra $\mathcal{X}^*\coloneqq (Up(\mathcal{X}),\cap,\cup,\to,\gets,\emptyset, X),$ where the implications are defined, for every $U,V \in Up(\X)$, as
\begin{align*}
U\to V&\coloneqq X\smallsetminus \down (U\smallsetminus V)=\big\{x\in X \colon \up x \cap U \subseteq V\big\},\\
U\gets V&\coloneqq \up(U\smallsetminus V)=\big\{x\in X \colon \down x \cap U \nsubseteq V\big\}.
\end{align*}

\begin{defi} \label{def bi-p-morphism}
Let $\X=(X, \leq)$ and $\Y=(Y,\leq)$ be posets. A map $f\colon X \to Y$ is called a \textit{bi-p-morphism}, denoted by $f\colon \X \to \Y$, if it satisfies the following conditions:
\benbullet
    \item \textbf{\textit{Order preserving:}} $\forall x,z\in X \; \big(x \leq z \implies f(x) \leq f(z)\big)$; \vspace{.1cm}
    
    \item \textbf{\textit{Up:}} $\forall x\in X, \forall y\in Y\; \big(f(x) \leq y \implies \exists z\in \up x\; (f(z)=y)\big)$; \vspace{.1cm}
    
    \item \textbf{\textit{Down:}} $\forall x\in X, \forall y\in Y\; \big(y \leq f(x) \implies\exists z\in \down x\; (f(z)=y)\big)$.
\ebullet
If moreover the map $f$ is surjective, $\Y$ is said to be a \textit{bi-p-morphic image} of $\X$ and we write $f \colon \X \twoheadrightarrow \Y$.
\end{defi}

\begin{Theorem}[Finite bi-Esakia Duality]\label{fin duality}
    The category of finite bi-Heyting algebras and their homomorphisms is dually equivalent to the category of finite posets and bi-p-morphisms. 
    In particular, if $\A$ and $\B$ are finite bi-Heyting algebras, then $\A$ embeds into $\B$ iff $\A_*$ is a bi-p-morphic image of $\B_*$.
\end{Theorem}

We further restrict this duality to bi-Gödel algebras. 
Recall that a \textit{co-forest} is a (possibly empty) disjoint union of co-trees.

\begin{Theorem}[{\cite[Thms. 3.1 \& 3.6]{Paper1}}] \label{si balgs}  
Let $\A\in \bivar^{< \omega}$. Then $\A$ is a $\balg$ iff $\A_*$ is a finite co-forest. Moreover, $\A$ is an SI $\balg$ iff $\A_*$ is a finite co-tree.
\end{Theorem}

The theory of \textit{subframe formulas} of bi-Gödel algebras was developed in \cite{Paper1, Martins01} (for an overview of these formulas and their use in superintuitionistic and modal logics we refer to \cite{Bezhan2} and \cite{Zakha}, respectively). 
With every finite and SI bi-Gödel algebra $\A$ we associate its subframe formula $\beta(\A)$, and if $\Y$ is a finite co-tree, we set $\beta(\Y) \coloneqq \beta(\Y^*)$.
The formula $\beta(\A)$ fully describes the $(\lor, \gets)$-reduct of $\A$, and its refutation in a bi-Gödel algebra $\B$ amounts to the existence of a $(\lor,\gets)$-embedding of $\A$ into some homomorphic image of $\B$.
Notably, subframe formulas govern the embeddability of finite co-trees into \textit{bi-Esakia co-forests} (i.e., co-forests equipped with a bi-Esakia topology).
The next result formalizes this for the finite case.

Recall that an order embedding is an order invariant map between posets and that we use the notation $\Y \hookrightarrow \X$ (resp. $\Y \not \hookrightarrow \X$) for the existence (resp. non-existence) of an order embedding from a poset $\Y$ into a poset $\X$.

\begin{Lemma}[Subframe Lemma {\cite[Lem. 4.24]{Paper1}}] \label{subframe lemma}
If $\X$ is a finite co-forest and $\Y$ a finite co-tree, then
\[
\X^* \models \beta(\Y) \iff \Y \not \hookrightarrow \X.
\]
\end{Lemma}

\section{Counting the locally finite varieties of bi-Gödel algebras}\label{Sec counting}

A variety $\V$ is \textit{locally finite} when its finitely generated members are finite.
If $\vdash$ is the logic algebraized by $\V$, then $\V$ is locally finite iff $\vdash$ is \textit{locally tabular}, i.e., in $\vdash$, there are only finitely many non-equivalent formulas in each finite number of propositional variables.

In the setting of \balg s, the notion of local finiteness is intrinsically connected with a particular family of co-trees, called the \textit{finite combs}: for each positive integer $n$, we define the $n$-\textit{comb} $\C_n\coloneqq (C_n,\leq)$ as the finite co-tree depicted in Figure \ref{Fig:finite-combs2}.

\begin{figure}[h]
\begin{tikzpicture}
    \tikzstyle{point} = [shape=circle, thick, draw=black, fill=black , scale=0.35]
    \node [label=right:{$c_1'$}] (1') at (1,0) [point] {};
    \node [label=left:{$c_1$}] (1) at (0.5,0.5) [point] {};
    \node [label=right:{$c_2'$}] (2') at (1.5,.5) [point] {};
    \node [label=left:{$c_2$}] (2) at (1,1) [point] {};
    \node [label=above:{$c_n$}] (n) at (1.75,1.75) [point] {};
    \node [label=right:{$c_n'$}] (n') at (2.25,1.25) [point] {};
    
    \draw (1)--(2);
    \draw (1')--(1);
    \draw (2')--(2);
    \draw (n')--(n);
    \draw [dotted] (2)--(n);
\end{tikzpicture}
\caption{The $n$-comb $\C_n$.}
\label{Fig:finite-combs2}
\end{figure}

In \cite[Thm.~5.1]{Paper1} (see Theorem \ref{big crit}), many statements related to the finite combs are shown to be equivalent to the local finiteness of an arbitrary variety of bi-Gödel algebras.
For our purposes, the following equivalence will suffice:

\begin{Theorem}\label{thm lf}
   A variety $\V$ of bi-Gödel algebras is locally finite iff $\V \models \beta(\C_n)$ for some $n \in \mathbb{Z}^+$.
\end{Theorem}

It follows that the locally finite varieties of \balg s are exactly those contained in the varieties of the form 
\[
\V_n \coloneqq \{\A \in \bg \colon \A \models \beta(\C_n) \},
\]
for $n \in \mathbb{Z}^+$.
Therefore, the number of locally finite subvarieties of $\bg$ is equal to
\[
|\bigcup_{n \in \omega} \Lambda(\V_n)|,
\]
where $|\Lambda(\V_n)|$ denotes the cardinality of $\Lambda(\V_n)$, the lattice of nontrivial subvarieties of $\V_n$.
We will show that the cardinality in the previous display is equal to $\aleph_0$.
That $\aleph_0 \leq |\bigcup_{n \in \omega} \Lambda(\V_n)|$ follows readily from the simple observation (stated in Lemma \ref{lem antichains equivalence}, but a glance at Figure \ref{Fig:finite-combs2} should be convincing enough) that given positive integers $n<m$, then $\C_n \hookrightarrow \C_m$ but $\C_m \not \hookrightarrow \C_n$, and therefore $\V_n \subsetneq \V_m$.

To prove the reverse inequality, it suffices to show that every $\Lambda(\V_n)$ is at most countable, as this implies
\[
|\bigcup_{n \in \mathbb{Z}^+} \Lambda(\V_n)| \leq \sum_{n \in \mathbb{Z}^+} |\Lambda(\V_n)| \leq \sum_{n \in \mathbb{Z}^+} \aleph_0 = \aleph_0.
\]
We achieve this by relying on the following result concerning \textit{Specht} varieties (i.e., varieties $\V$ whose subvarieties are all finitely axiomatizable):

\begin{Theorem}[{\cite[Thm.~6.21]{CITKIN_2020}}]\label{specht}
    If $\V$ is a locally finite, finitely axiomatizable, and congruence distributive variety, then the following conditions are equivalent:
    \benroman
        \item $\V$ is a Specht variety;
        \item $|\Lambda(\V)|\leq \aleph_0$;
        \item the poset $(\V_{SI}^{< \omega}, \HHH \SSS)$ has no infinite antichains;
        \item the lattice $\Lambda(\V)$ has no infinite descending chains.
    \eroman 
\end{Theorem}

Some comments are in order.
Firstly, to see that the varieties $\V_n$ actually fall under the conditions of the above statement, recall that they are: locally finite by Theorem \ref{thm lf};  finitely axiomatizable because  so is $\cbipc$ and we have
\[
\vdash_{\V_n} = \cbipc + (p \to q) \lor (q \to p) + \beta(\C_n);
\]
and congruence distributive because, by Proposition \ref{props of bg}.(ii), so is the variety $\bg$.

Secondly, we need to justify why the relation $\HHH \SSS$ is a partial order on $(\V_n)_{SI}^{ < \omega}$, the set of SI members of $\V_n$ which are finite.
Let $\A$ and $\B$ be two such members and recall that
\[
\A \HHH \SSS \B \iff \A \in \HHH \SSS \{\B\},
\]
that is, $\A$ is a homomorphic image of a subalgebra of $\B$.
Since Proposition \ref{props of bg}.(iii) ensures that any subalgebra of the SI bi-Gödel algebra $\B$ will also be SI, and condition (i) of the same result tells us that SI \balg s are always simple (hence have no nontrivial homomorphic images), we conclude that $\A \HHH \SSS \B$ iff $\A$ embeds into $\B$.
It is now immediate that $\HHH \SSS$ is both reflexive and transitive, while antisymmetry comes from noting that the algebras in $(\V_n)_{SI}^{ < \omega}$ are all finite.
Moreover, using Theorem \ref{fin duality} we know that $\A$ embeds into $\B$ iff $\A_*$ is a bi-p-morphic image of $\B_*$, a relation that henceforth is denoted by $\A_* \leq_p \B_*$ and called the \textit{bi-p-morphic image relation}.

Summarizing, we have established that for $\A, \B \in (\V_n)_{SI}^{ < \omega}$, the following equivalences hold:
\[
\A \HHH \SSS \B \iff \A \text{ embeds into } \B \iff \A_* \text{ is a bi-p-morphic image of } \B_* \iff \A_* \leq_p \B_*. 
\]

As previously stated, our goal is to show that every variety   
\[
\V_n = \{\A \in \bg \colon \A \models \beta(\C_n) \}
\]
is Specht.
We do so by showing that $(\V_n)_{SI}^{ < \omega}$ has no infinite antichains w.r.t. the $\HHH \SSS$ relation, which suffices because of Theorem \ref{specht}.
By the equivalences displayed above, we can achieve this by proving that the set of bi-Esakia duals of members of $(\V_n)_{SI}^{ < \omega}$ has no infinite antichains w.r.t. to the bi-p-morphic image relation $\leq_p$.
And since Theorem \ref{si balgs} and Lemma \ref{subframe lemma} ensure that
\begin{align*}
    \A \in (\V_n)_{SI}^{ < \omega} & \iff \A_* \text{ is a finite co-tree s.t. } \A \models \beta(\C_n) \\
    & \iff \A_* \text{ is a finite co-tree s.t. } \C_n \not \hookrightarrow \A_*\\
    & \iff \A_* \text{ is a finite co-tree which does not admit $\C_n$ as a subposet,}
\end{align*}
the aforementioned set of duals of $(\V_n)_{SI}^{ < \omega}$ can be described as
\[
\T_n \coloneqq \{\X \colon \X \text{ is a finite co-tree s.t. } \C_n \not \hookrightarrow \X \}.
\]
We recall that we identify order isomorphic finite co-trees, and note that by the previous discussion, the next result is now immediate.

\begin{Lemma}\label{lem antichains equivalence}
    Let $n$ and $m$ be positive integers such that $n < m$. 
    \benroman
        \item The poset $\big((\V_n)_{SI}^{< \omega}, \HHH \SSS\big)$ is order isomorphic to $(\T_n, \leq_p)$.

        \item The poset $\big((\V_n)_{SI}^{< \omega}, \HHH \SSS\big)$ has no infinite antichains iff the poset $(\T_n, \leq_p)$ has no infinite antichains.

        \item $\T_n \subsetneq \T_m$ and $\V_n \subsetneq \V_m$.
    \eroman
\end{Lemma}

The remainder of this section is dedicated to the proof of Proposition \ref{prop tn is bpo}, which states that every $(\T_n, \leq_p)$ is a BPO.
Since BPOs cannot have infinite antichains, using the previous lemma together with Proposition \ref{specht}, we will not only obtain that
\[
\aleph_0 \leq |\bigcup_{n \in \mathbb{Z}^+} \Lambda(\V_n)| \leq \sum_{n \in \mathbb{Z}^+} |\Lambda(\V_n)| \leq \sum_{n \in \mathbb{Z}^+} \aleph_0 = \aleph_0,
\]
but also that every subvariety of every $\V_n$ is finitely axiomatizable, thus establishing the main result of this paper:

\begin{Theorem}\label{main thm}
    There are only $\aleph_0$ locally finite varieties of bi-Gödel algebras, all of which are finitely axiomatizable.
\end{Theorem}

Let us first introduce some definitions regarding posets and some properties of their bi-p-morphisms that will help us in what follows.
Take a poset $\X$ and $x,y \in X$.
If $y \leq z \leq x$ implies either $z=y$ or $z=x$, we write $y \prec x$ and say that $y$ is an \textit{immediate predecessor} of $x$, or that $x$ is an \textit{immediate successor} of $y$.
Moreover, we denote the set of immediate predecessors of $x$ by 
\[
\precc x \coloneqq \{ z \in X \colon z \prec x\}.
\]
We will also make use of the following notations:
\[
\upc x \coloneqq \up x\smallsetminus\{x\} \hspace{.3cm} \text{ and } \hspace{.3cm} \downc x \coloneqq \down x \smallsetminus\{x\}.
\]

\begin{Lemma}\label{lem bi-p-morphs}
    Let $f\colon \X \to \Y$ be a bi-p-morphism between posets. 
    \benroman
        \item $f[\up x]= \up f(x)$ and $f[\down x]=\down f(x)$, for all $x \in X$.
        
        \item $f$ always maps maximal (resp. minimal) points of $\X$ to maximal (resp. minimal) points of $\Y$.

        \item  If $\X$ and $\Y$ are co-trees then the co-root of $\X$ must be mapped to the co-root of $\Y$, the map $f$ is necessarily surjective, and for $x,z\in X$,
        \[
        x \prec z \text{ implies either } f(x)=f(z) \text{ or } f(x) \prec f(z).
        \]
    \eroman
\end{Lemma}



The first step in our proof that the bi-p-morphic image relation $\leq_p$ on 
\[
\T_n = \{\X \colon \X \text{ is a finite co-tree s.t. } \C_n \not \hookrightarrow \X \}
\]
is a BPO is to find a useful characterization of these co-trees.

Recall Figure \ref{Fig:finite-combs2}.
Since the $1$-comb $\C_1$ is just a two element chain, it is easy to see that the only finite co-tree that does not admit $\C_1$ as a subposet is the trivial singleton poset, which we denote by $\bullet$.
So $\T_1 = \{\bullet\}$ and $\leq_p$ is obviously a well-order on $\T_1$, hence a BPO by Theorem \ref{thm bpo generation}.(i).

Next, we show that the co-trees in $\T_2 = \{\X \colon \X \text{ is a finite co-tree s.t. } \C_2 \not \hookrightarrow \X \}$ are exactly those of the form depicted in Figure \ref{Fig:t2}, where $m,k \in \omega$ and the set $\precc x_m=\{y_0,\dots, y_k\}$ of immediate predecessors of $x_m$ coincides with the set of minimal elements of the co-tree $\tau(m,k)$.
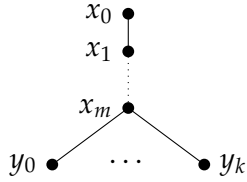
\begin{figure}[h]
\begin{tikzpicture}
    \tikzstyle{point} = [shape=circle, thick, draw=black, fill=black , scale=0.35]
    \node [label=left:{$x_0$}] (0) at (0,3) [point] {};
    \node [label=left:{$x_1$}] (1) at (0,2.5) [point] {};
    \node [label=left:{$x_m$}] (m) at (0,1.75) [point] {};
    \node [label=left:{$y_0$}] (0') at (-1,1) [point] {};
    \node [label=right:{$y_k$}] (k) at (1,1) [point] {};
    \node [label=above:{$\dots$}] () at (0,.75) [] {};

    \draw (0)--(1);
    \draw [dotted] (1)--(m);
    \draw (0')--(m)--(k);
\end{tikzpicture}
\caption{The co-tree $\tau(m,k)$.}
\label{Fig:t2}
\end{figure} 

\noindent In fact, we claim that not only is the map $\tau \colon \omega \times \omega \to \T_2$ a bijection, but it induces an order isomorphism between $(\omega, \leq) \times (\omega, \leq)$ (where $\leq$ is the natural order of $\omega$) and $(\T_2,\leq_p)$.

\begin{Proposition}\label{prop tau is iso}
    The map $\tau \colon (\omega, \leq) \times (\omega, \leq) \to (\T_2,\leq_p)$ is an order isomorphism (where $\tau(m,k)$ is the co-tree depicted in Figure \ref{Fig:t2}).
    Consequently, $(\T_2,\leq_p)$ is a BPO.
\end{Proposition}
\begin{proof}
    We start by proving that $\tau$ is indeed a bijection.
    Let $m,k\in \omega$.
    From a quick inspection of the structure of $\C_2$ and of the co-tree $\tau(m,k)$, it is immediate that there can be no order embedding from $\C_2$ into $\tau(m,k)$, i.e., $\C_2 \not \hookrightarrow \tau(m,k)$, and since $\tau(m,k)$ is a finite co-tree by definition, we have that $\tau(m,k) \in \T_2 = \{\X \colon \X \text{ is a finite co-tree s.t. } \C_2 \not \hookrightarrow \X \}$.
    So the map $\tau$ is well defined.

    To establish surjectivity, take a finite co-tree $\X$ and suppose $\C_2 \not \hookrightarrow \X$.
    Denote the co-root of $\X$ by $x_0\coloneqq Max(\X)$.
    If $\X$ is a chain, then it is a finite chain, so we can list $X=\{x_0, \dots, x_m\}$, where $m \in \omega$ and $x_0 > \dots > x_m$. 
    We set $k \coloneqq 0$ and $y_0 \coloneqq x_m$, hence $\X=\tau(m,0)$ follows. 

    Assume now that $\X$ is not a chain. 
    Because of this, and since $\X$ has a maximum $x_0$, an easy argument proves that there must exist a greatest element $z \in \down x_0=X$ with more than one immediate predecessor.
    Thus, we have $\precc z = \{y_0, \dots , y_k\}$ for some $1 \leq k \in \omega$.
    Furthermore, the definition of $z$ ensures that $\down x_0 \smallsetminus \downc z$ is a finite chain, so we can list $\down x_0 \smallsetminus \downc z =\{x_0, \dots, x_m\}$, where $m \in \omega$ and $x_0 > \dots > x_m = z$. 
    Notice that
    \[
    X= \down x_0 = (\down x_0 \smallsetminus \downc x_m) \cup \downc x_m= (\down x_0 \smallsetminus \downc x_m) \cup \down (\precc x_m)= \{x_0, \dots, x_m\} \cup \down \{y_0, \dots , y_k\}
    \]

    In order to obtain the equality $\X=\tau(m,k)$, and thus prove that $\tau$ is surjective, it now suffices to show that all the $y_0, \dots , y_k$ are minimal points.
    For suppose otherwise and say, without loss of generality, that there exists $v < y_0$.
    Then recalling that $z=x_m$ has at least one other immediate predecessor $y_1$ distinct from $y_0$ ensures that the subposet of $\X$ generated by $\{v,y_0,y_1,z\}$ is a copy of the $2$-comb $\C_2$, thus contradicting $\C_2 \not \hookrightarrow \X$.

    Having shown that the map $\tau$ is surjective, to conclude that it is an order isomorphism it remains to prove that $\tau$ is order invariant, i.e., that 
    \[
    (m,k) \leq (m',k') \iff \tau(m,k) \leq_p \tau(m',k')
    \]
    holds true for all $m,m',k,k' \in \omega$.
    Accordingly, take $m,m',k,k' \in \omega$ and list 
    \[
    \tau(m',k') = \{x_0', \dots , x'_{m'}\} \cup \{y'_0, \dots, y'_{k'}\} \text{ and } \tau(m,k) = \{x_0, \dots , x_{m}\} \cup \{y_0, \dots, y_{k}\}
    \]
    following the naming conventions of Figure \ref{Fig:t2}.
    
    Suppose first that $(m,k) \leq (m',k')$, i.e., that $m \leq m'$ and $k \leq k'$.
    We define a surjective map $f \colon \tau(m',k') \twoheadrightarrow \tau(m,k)$ by setting
    \[
    f(x'_i) \coloneqq 
    \begin{cases}
        x_i & \text{if } i \leq m, \\
        x_m & \text{if } m\leq i \leq m',
    \end{cases}
    \hspace{.3cm}
    \text{ and } 
    \hspace{.3cm}
    f(y'_i) \coloneqq 
    \begin{cases}
        y_i & \text{if } i \leq k, \\
        y_k & \text{if } k\leq i \leq k',
    \end{cases}
    \]
    and an easy verification shows it to be a bi-p-morphism (see Definition \ref{def bi-p-morphism}), thus $\tau(m,k) \leq_p \tau(m',k')$.
    This proves that $\tau$ is order preserving.

    To see that $\tau$ is also order reflecting, hence order invariant, we now assume $\tau(m,k) \leq_p \tau(m',k')$, i.e., that there exists a surjective bi-p-morphism $g\colon \tau(m',k') \twoheadrightarrow \tau(m,k)$.
    We know from Lemma \ref{lem bi-p-morphs}.(ii) that $g[\{y'_0, \dots, y'_{k'}\}] \subseteq \{y_0, \dots, y_{k}\}$.
    If this inclusion is proper, i.e., if there exists $y_j \in \{y_0, \dots, y_{k}\} \smallsetminus g[\{y'_0, \dots, y'_{k'}\}]$, then the fact that $g$ is surjective forces the existence of $i \leq m'$ satisfying $g(x_i')=y_j$. 
    But bi-p-morphisms are order preserving, so $\{y'_0, \dots, y'_{k'}\} \subseteq \down x_i'$ implies that all the points in $g[\{y'_0, \dots, y'_{k'}\}]$ must lie below $y_j$.
    Since we assumed $y_j \in \{y_0, \dots, y_{k}\} \smallsetminus g[\{y'_0, \dots, y'_{k'}\}]$, this contradicts the minimality of $y_j$.
    Thus, we have $g[\{y'_0, \dots, y'_{k'}\}] = \{y_0, \dots, y_{k}\}$ and it is clear that $k\leq k'$.
    Again using the surjectivity of $g$, together with our listing of $\tau(m,k)$ and $\tau(m',k')$ displayed above, it is easy to see that the equality $g[\{y'_0, \dots, y'_{k'}\}] = \{y_0, \dots, y_{k}\}$ entails $g[\{x_0', \dots , x'_{m'}\}]=\{x_0, \dots , x_{m}\}$, so $m\leq m'$ follows and we are done showing that $(m,k) \leq (m',k')$.

    The last part of the statement is now immediate from Theorem \ref{thm bpo generation}.(i)\&(iii), since we have proved $(\T_2,\leq_p)$ to be order isomorphic to the product of two well-orders.
\end{proof}

We now proceed to generalize this useful depiction of the co-trees in $\T_2$ to an arbitrary $\T_{n+1}$.
We do so by noting that in the previous proof, demanding the minimality of the $y_0, \dots, y_k$ (the points below $z=x_m$, the greatest element of $\X$ with more than one immediate predecessor) is equivalent to forbidding the existence of an order embedding from the $1$-comb $\C_1$ into any of the co-trees $\down y_0 \dots, \down y_k$.
In other words, we are requiring that all the co-trees strictly below $x_m$ are elements of $\T_1$.

We claim that any $\X \in \T_{n+1} = \{\X \colon \X \text{ is a finite co-tree s.t. } \C_{n+1} \not \hookrightarrow \X \}$ is either a chain, or consists of a finite chain 
\[
Max(\X) =\colon x_0 > \dots > x_m \coloneq Max(\{z \in X \colon |\precc z|>1\})
\]
and of $|\precc x_m|$-many co-trees in $\T_n$, whose co-roots are exactly the immediate predecessors of $x_m$.
We depict this in Figure \ref{Fig:structure}.


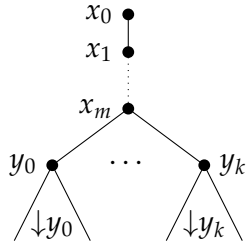
\begin{figure}[h]
\begin{tikzpicture}
    \tikzstyle{point} = [shape=circle, thick, draw=black, fill=black , scale=0.35]
    \node [label=left:{$x_0$}] (0) at (0,3) [point] {};
    \node [label=left:{$x_1$}] (1) at (0,2.5) [point] {};
    \node [label=left:{$x_m$}] (m) at (0,1.75) [point] {};
    \node [label=left:{$y_0$}] (0') at (-1,1) [point] {};
    \node [label=right:{$y_k$}] (k) at (1,1) [point] {};
    \node [label=above:{$\down y_0$}] () at (-1,-.25) [] {};
    \node [label=above:{$\down y_k$}] () at (1,-.25) [] {};
    \node [label=above:{$\dots$}] () at (0,.75) [] {};

    \draw (0)--(1);
    \draw [dotted] (1)--(m);
    \draw (0')--(m)--(k);
    \draw (-1.5,0)--(0')--(-.5,0);
    \draw (.5,0)--(k)--(1.5,0);
\end{tikzpicture}
\caption{A co-tree in $\T_{n+1}$.}
\label{Fig:structure}
\end{figure} 

Before we prove our claim, we note that a finite comb always order embeds into a bigger finite comb, that is, $\C_n \hookrightarrow \C_m$ for all $1\leq n \leq m \in \omega$.
Consequently, given a finite co-tree $\X$, we have that $\C_n \not \hookrightarrow \X$ implies $\C_m \not \hookrightarrow \X$ whenever $n \leq m$.
This establishes the condition
\[
\forall 1 \leq n \leq m \in \omega \, (\T_n \subseteq \T_m).
\]

\begin{Lemma}\label{lem structure}
    Let $\X$ be a finite co-tree and $n$ a positive integer.
    Then $\X \in \T_{n+1}$ iff $\X$ is of the form depicted in Figure \ref{Fig:structure}, where $m,k\in \omega$ and $\down y_i \in \T_n$ for all $i\leq k$.
\end{Lemma}
\begin{proof}
    We start by proving the left to right implication.
    Take $\X$ and suppose $\X \in \T_{n+1}$, i.e., that $\X$ is a finite co-tree such that $\C_{n+1} \not \hookrightarrow \X$.
    We denote the co-root of $\X$ by $x_0\coloneqq Max(\X)$.
    If $\X$ is a chain, then it is a finite chain, so we can list $X=\{x_0, \dots, x_m\}$, where $m \in \omega$ and $x_0 > \dots > x_m$. 
    We set $k \coloneqq 0$ and $y_0 \coloneqq x_m$ and we are done with this case, because 
    \[
    \down y_0 = \{y_0\} \in \{\bullet\} = \T_1 \subseteq \T_n.
    \]

    Assume now that $\X$ is not a chain. 
    Because of this, and since $\X$ has a maximum $x_0$, an easy argument proves that there must exist a greatest element $z \in \down x_0=X$ with more than one immediate predecessor.
    Thus, we have $\precc z = \{y_0, \dots , y_k\}$ for some $1 \leq k \in \omega$.
    Furthermore, the definition of $z$ ensures that $\down x_0 \smallsetminus \downc z$ is a finite chain, so we can list $\down x_0 \smallsetminus \downc z =\{x_0, \dots, x_m\}$, where $m \in \omega$ and $x_0 > \dots > x_m \coloneqq z$. 
    Notice that
    \[
    X= \down x_0 = (\down x_0 \smallsetminus \downc x_m) \cup \downc x_m= (\down x_0 \smallsetminus \downc x_m) \cup \down (\precc x_m)= \{x_0, \dots, x_m\} \cup \down \{y_0, \dots , y_k\}
    \]
    
    It remains to show that $\down y_i \in \T_n$ for all $i \leq k$.
    To prove this, we recall $1 \leq k$ and take $i < j \leq k$.
    Suppose that $\down y_i \notin \T_n$.
    It is clear that $\down y_i$ is a finite co-tree since it is a subposet of the finite co-tree $\X$, so the assumption $\down y_i \notin \T_n$ entails the existence of an order embedding  $f \colon \C_n  \hookrightarrow \down y_i$.
    We recall Figure \ref{Fig:finite-combs2} and identify $\C_n$ with the subposet $\down c_n$ of $\C_{n+1}$.
    Defining $g \colon \C_{n+1} \to \X$ by 
    \[
    g{\restriction_{\C_n}} \coloneqq f \hspace{.2cm}\text{ and }\hspace{.2cm} g(c_{n+1}) \coloneqq x_m \hspace{.2cm}\text{ and }\hspace{.2cm} g(c'_{n+1})\coloneqq y_j
    \]
    clearly yields an order embedding of $\C_{n+1}$ into $\X$, thus contradicting $\X \in \T_{n+1}$.
    We conclude $\down y_i \in \T_n$, as desired.
    
    Next we prove that any finite co-tree $\X$ of the form depicted in Figure \ref{Fig:structure} must be contained in $\T_{n+1}$.
    For suppose not, so there exists an order embedding $f \colon \C_{n+1} \hookrightarrow \X$.
    Consider the point $f(c_{n+1}) \in X$ and recall that $\C_{n+1} = \down c_{n+1}$.
    It follows that for any $i \leq k$ we must have $f(c_{n+1}) \nleq y_i$, since otherwise the fact that $f$ is an order embedding would force $f[\C_{n+1}]=f[\down c_{n+1}] \subseteq \down y_i$, hence $f \colon \C_{n+1} \hookrightarrow \down y_i$ and we would have $\down y_i \notin \T_{n+1}$, contradicting the assumption that $\down y_i \in \T_n$ because $\T_n \subseteq \T_{n+1}$.

    We now know that $f(c_{n+1}) \nleq y_i$ for all $i \leq k$, and since the equality 
    \[
    X = \up x_m \cup \down y_0 \cup \dots \cup \down y_k
    \] 
    is immediate from the depiction of $\X$, we can infer that $f(c_{n+1}) \geq x_m$.
    As $\up x_m = \{x_0, \dots, x_m\}$, we have $f(c_{n+1})=x_t$ for some $t \leq m$.
    We claim that this is enough to prove that $f(c_n) \leq y_j$ for some $j \leq k$, hence also $f[\down c_n] = \C_n \hookrightarrow \down y_j$, which again contradicts $\down y_j \in \T_n$.
    To see why, recall $c_n < c_{n+1} > c'_{n+1}$ and that $c_n$ is incomparable to $c'_{n+1}$.
    So $f$ being an order embedding implies $f(c_n) < f(c_{n+1}) = x_t > f(c'_{n+1})$ and that $f(c_n)$ is incomparable to $f(c'_{n+1})$.  
    But then $f(c_n)$ cannot lie in the chain $\{x_t, \dots ,x_m\}$, since 
    \[
    f(c'_{n+1}) \in \down x_t = \{x_t, \dots ,x_m\} \cup \down x_m
    \]
    would force $f(c_n)$ and $f(c'_{n+1})$ to be comparable.
    As we also have 
    \[
    f(c_{n}) \in \down x_t = \{x_t, \dots ,x_m\} \cup \down x_m,
    \]
    from $f(c_{n}) \notin \{x_t, \dots ,x_m\}$ we can now conclude $f(c_{n}) \in  \downc x_m = \down y_0 \cup \dots \cup \down y_k$, as desired.
\end{proof}

With this characterization of the co-trees in $\T_{n+1}$, we can now introduce the strategy behind our proof that every such set equipped with the bi-p-morphic image relation $\leq_p$ is a BPO.
Say we have a co-tree $\X$ in $\T_{n+1}$.
By Lemma \ref{lem structure}, we know $\X$ to be of the form
\begin{figure}[h]
\begin{tikzpicture}
    \tikzstyle{point} = [shape=circle, thick, draw=black, fill=black , scale=0.35]
    \node [label=left:{$x_0$}] (0) at (0,3) [point] {};
    \node [label=left:{$x_1$}] (1) at (0,2.5) [point] {};
    \node [label=left:{$x_m$}] (m) at (0,1.75) [point] {};
    \node [label=left:{$y_0$}] (0') at (-1,1) [point] {};
    \node [label=right:{$y_k$}] (k) at (1,1) [point] {};
    \node [label=above:{$\down y_0$}] () at (-1,-.25) [] {};
    \node [label=above:{$\down y_k$}] () at (1,-.25) [] {};
    \node [label=above:{$\dots$}] () at (0,.75) [] {};

    \draw (0)--(1);
    \draw [dotted] (1)--(m);
    \draw (0')--(m)--(k);
    \draw (-1.5,0)--(0')--(-.5,0);
    \draw (.5,0)--(k)--(1.5,0);
\end{tikzpicture}
\end{figure} 

\noindent for some $m,k \in \omega$ and some $\Y_i \coloneqq \down y_i \in \T_n$.
We can write this as
\[
\X = \{x_0,\dots,x_m\} \cup \Y_0 \cup \dots \cup \Y_k = \{x_0,\dots,x_m\} \cup \{y_0, \dots, y_k\} \cup \down y_0 \cup \dots \cup \down y_k.
\]
So $\X$ consists of an upper part, the co-tree 
\[
\{x_0,\dots,x_m\} \cup \{y_0, \dots, y_k\}=\tau(m,k) \in \T_2
\]
(recall Proposition \ref{prop tau is iso}), and a lower part, a finite collection of elements of $\T_n$, whose order is irrelevant (the co-tree obtained by switching $y_0$ and $\down y_0$ with $y_k$ and $\down y_k$ is still $\X$ because we identify isomorphic co-trees) and where the same element can appear multiple times (nothing forbids the co-trees $\down y_0$ and $\down y_k$ to be isomorphic).
In other words, the lower part of $\X$ can be described as the finite multiset $[\Y_0, \dots, \Y_k]$ of $\T_n$.
Moreover, these upper and lower parts completely determine the co-tree $\X$.
Accordingly, we can identify $\X$ with the pair
\[
\pi(\X)\coloneqq \big(\tau(m,k),[\Y_0, \dots, \Y_k]\big)
\]
and this assignment clearly defines an injection of sets $\pi \colon \T_{n+1} \to \T_2 \times \T_n^\#$, where $\T_n^\#$ denotes the set of finite multisets of $\T_n$.

Notice that $\pi$ is not surjective, because a pair $\big(\tau(m,k),M\big) \in \T_2 \times \T_n^\#$ for which the length of the finite multiset $M$ does not coincide with the integer $k+1$ will never correspond under our assignment to a co-tree in $\T_{n+1}$.

Thus, if we can show that $\pi \colon \T_{n+1} \to \T_2 \times \T_n^\#$ induces an order reflecting map from $(\T_{n+1}, \leq_p)$ to $(\T_2, \leq_p) \times (\T_n^\#, <<)$ (where $<<$ is the multiset projectivity relation, see Definition \ref{def multi proj}), and we combine Theorem \ref{thm bpo generation}.(ii)\&(iii) with Theorem \ref{thm multi proj is bpo} and with Proposition \ref{prop tau is iso}, the only thing missing in the proof that $(\T_{n+1}, \leq_p)$ is a BPO will be the assumption that so is $(\T_{n}, \leq_p)$.

\begin{Lemma}\label{lem ord reflecting}
    The map $\pi \colon (\T_{n+1}, \leq_p) \to (\T_2, \leq_p) \times (\T_n^\#, <<)$ defined above is order reflecting.  
\end{Lemma}
\begin{proof}
    Let $\X,\X'\in \T_{n+1}$ be such that
    \[
    \pi(\X) = \big(\tau(m,k),[\Y_0, \dots, \Y_k]\big) \hspace{.3cm} \text{ and } \hspace{.3cm} \pi(\X')=\big(\tau(m',k'),[\Y'_0, \dots, \Y'_{k'}]\big).
    \]
    We need to prove that
    \[
    \pi(\X) \leq_p \times << \pi(\X')\hspace{.3cm} \text{ implies } \hspace{.3cm}\X \leq_p \X',
    \]
    i.e., that
    \[
    \tau(m,k) \leq_p \tau(m',k') \hspace{.3cm} \text{ and } \hspace{.3cm}[\Y_0, \dots, \Y_k] << [\Y'_0, \dots, \Y'_{k'}] \hspace{.3cm}\text{ imply }\hspace{.3cm}\X \leq_p \X'.
    \]
    To this end, we assume the antecedent of the previous display.
    From $\tau(m,k) \leq_p \tau(m',k')$ we know there exists a surjective bi-p-morphism $g\colon \tau(m',k') \twoheadrightarrow \tau(m,k)$, while $[\Y_0, \dots, \Y_k] << [\Y'_0, \dots, \Y'_{k'}]$ ensures the existence of a surjective map $h \colon [\Y'_0, \dots, \Y'_{k'}] \twoheadrightarrow [\Y_0, \dots, \Y_k]$ satisfying $h(\Y'_i) \leq_p \Y'_i$ for all $i\leq k'$ (see Definition \ref{def multi proj}).
    Hence, for every $i\leq k'$ there is a surjective bi-p-morphism $h_i \colon \Y'_i \twoheadrightarrow h(\Y'_i)$.
    We define $f\colon \X' \to \X$ by
    \[
    f(x)\coloneqq \begin{cases}
        g(x) & \text{ if } x\in \tau(m',k'),\\
        h_i(x) & \text{ if } x\in \Y'_i.
    \end{cases}
    \]
    A straightforward verification shows that $f$ is a surjective bi-p-morphism, hence $\X \leq_p \X'$ follows, as desired.
\end{proof}

\begin{Remark}
    While the map $\pi \colon (\T_{n+1}, \leq_p) \to (\T_2, \leq_p) \times (\T_n^\#, <<)$ is both injective and order reflecting, it is an order embedding only when $n=1$ (in which case it is actually an order isomorphism).
    For example, let $n >1$ and consider the co-trees depicted in Figure \ref{Fig:counter ex}. 
    It is easy to see that $\X \leq_p \X'$, but it is not the case that
    \[ 
    \big(\tau(1,1),[\tau(0,0),\tau(0,0)]\big) \leq_p \times << \big(\tau(0,1),[\tau(0,1),\tau(0,1)]\big),
    \]
    because $\tau(1,1) \nleq_p \tau(0,1)$.
    We can, however, always obtain an order embedding by considering the restricted map
    \[
    \pi\restriction \colon (\T_{n+1}\smallsetminus \T_n, \leq_p) \to (\T_2, \leq_p) \times (\T_n^\#, <<).
    \]
    Since we will not need this fact, we omit its proof as it is fairly lengthy and technical.
\end{Remark}

\begin{figure}[h]
\centering
\begin{tabular}{cc}
\begin{tikzpicture}
    \tikzstyle{point} = [shape=circle, thick, draw=black, fill=black , scale=0.35]

    \node [] (0) at (0,2) [point] {};
    \node [] (1) at (0,1) [point] {};
    \node [] (0') at (-.5,0) [point] {};
    \node [] (1') at (.5,0) [point] {};
    \node [label=below:{$\X\coloneqq \big(\tau(1,1),[\tau(0,0),\tau(0,0)]\big)$}] () at (0,0) [] {};

    \draw (0')--(1)--(1');
    \draw (1)--(0);
\end{tikzpicture}
\hspace{3cm}
\begin{tikzpicture}
    \tikzstyle{point} = [shape=circle, thick, draw=black, fill=black , scale=0.35]
    \node [] (a) at (0,1.5) [point] {};
    \node [] (b) at (-.5,.75) [point] {};
    \node [] (c) at (.5,.75) [point] {};
    \node [] (d) at (-.75,0) [point] {};
    \node [] (e) at (-.25,0) [point] {};
    \node [] (f) at (.25,0) [point] {};
    \node [] (g) at (.75,0) [point] {};
    \node [label=below:{$\X'\coloneqq \big(\tau(0,1),[\tau(0,1),\tau(0,1)]\big)$}] () at (0,0) [] {};
    
    \draw (d)--(b)--(e);
    \draw (f)--(c)--(g);
    \draw (b)--(a)--(c);

\end{tikzpicture}

\end{tabular}
\caption{The co-trees $\X,\X' \in \T_{n+1}$.}
\label{Fig:counter ex}
\end{figure}
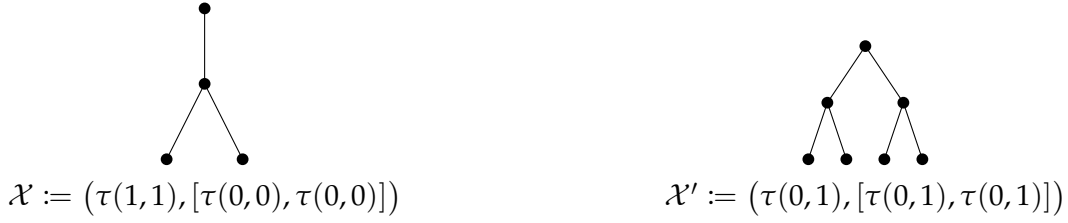

\begin{Proposition} \label{prop tn is bpo}
    For every positive integer $n$, the poset $(\T_n, \leq_p)$ is a BPO.
\end{Proposition}
\begin{proof}
    We use induction on $n$, noting that the base case is trivial because $\T_1=\{\bullet\}$, while Proposition \ref{prop tau is iso} already proves that $(\T_2, \leq_p)$ is a BPO.

    Suppose that $(\T_n, \leq_p)$ is a BPO for some $n\geq 2$.
    By Theorem \ref{thm multi proj is bpo}, we know that the set $\T_n^\#$ of finite multisets of $\T_n$ equipped with the multiset projectivity relation $<<$ must also be a BPO.
    And since $(\T_2, \leq_p)$ is a BPO by above, Theorem \ref{thm bpo generation}.(iii) ensures that the product $(\T_2, \leq_p) \times (\T_n^\#, <<)$ is a BPO as well.
    Furthermore, Lemma \ref{lem ord reflecting} establishes the existence of an order reflecting map from the poset $(\T_{n+1}, \leq_p)$ to the BPO $(\T_2, \leq_p) \times (\T_n^\#, <<)$.
    Thus, Theorem \ref{thm bpo generation}.(ii) entails that $(\T_{n+1}, \leq_p)$ is a BPO.    
\end{proof}

\noindent\textit{Proof of Theorem \ref{main thm}.}
We want to show that there are only $\aleph_0$ locally finite varieties of bi-Gödel algebras, all of which are finitely axiomatizable.

By Theorem \ref{thm lf}, a variety $\V \subseteq \bg$ is locally finite iff $\V \subseteq \V_n= \{\A \in \bg \colon \A \models \beta(\C_n)\}$ for some $n \in \omega$.
It follows that the number of locally finite varieties of bi-Gödel algebras is given by
\[
\sum_{n \in \mathbb{Z}^+} |\Lambda(\V_n)|,
\]
where $|\Lambda(\V_n)|$ denotes the cardinality of the lattice of nontrivial subvarieties of $\V_n$.

Take an arbitrary positive integer $n$ and recall that Lemma \ref{lem antichains equivalence} ensures that $\big((\V_n)_{SI}^{< \omega}, \HHH \SSS\big)$ has no infinite antichains iff the poset $(\T_n, \leq_p)$ has no infinite antichains.
We know from Proposition \ref{prop tn is bpo} that $(\T_n, \leq_p)$ is a BPO, i.e., it has no bad super sequences, and in particular, no infinite antichains.
Hence, neither does $\big((\V_n)_{SI}^{< \omega}, \HHH \SSS\big)$ by the previous equivalence.
Consequently, Theorem \ref{specht} entails that  $|\Lambda(\V_n)| \leq \aleph_0$.
Thus, we can conclude that
\[
\aleph_0 \leq \sum_{n \in \mathbb{Z}^+} |\Lambda(\V_n)| \leq \sum_{n \in \mathbb{Z}^+} \aleph_0 = \aleph_0.
\]
Furthermore, the same theorem also ensures that every $\V_n$ is a Specht variety (i.e., all of its subvarieties are finitely axiomatizable), so our criterion for local finiteness proves the last part of the statement. \qed

\begin{Theorem}\label{thm depiction}
    The depiction of $\Lambda(\bg)$ in Figure \ref{Fig:the lattice} is accurate.
\end{Theorem}
\begin{proof}
    Given a bi-Heyting algebra $\A$ and a finite poset $\X$, recall the notations $\mathbb{V}(\A)\coloneqq \HHH\SSS\PPP \{\A\}$ and $\mathbb{V}(\X) \coloneqq \mathbb{V}(\X^*)$.


    We first show that the equalities we claimed to be true in our depiction of $\Lambda(\bg)$ and in the subsequent comments indeed hold.
    Recall the definitions of the finite combs (possibly with handle, see Figure \ref{Fig:finite-hcombs2}) and of the finite chains $\La_n$, which yield
    \[
    \C'_0= \{\bullet\}= \La_1 \hspace{.3cm} \text{ and } \hspace{.3cm} \C_1=\La_2.
    \]
    Hence $\mathbb{V}(\C_1)=\mathbb{V}(\La_2)$, and it is well known that the dual of $\{\bullet\}$ generates the variety of Boolean algebras, thus
    \[\mathbb{V}(\C'_0)=\mathsf{BA}=\mathbb{V}(\La_1)=\V_1 = \{\A \in \bg\colon \A \models \beta(\C_1)\} = \{\A \in \bg\colon \C_1 \not \hookrightarrow \A_*\}.\]

    The partition $\Lambda(\bg) = \up \V_{FC} \biguplus \bigcup_{n\in \ZZ^+} \down \V_n$ is immediate from the criterion for local finiteness of varieties of bi-Gödel algebras found in Theorem \ref{The big theorem}.(2) together with the definition of 
    \[
    \V_n =\{\A \in \bg\colon \A \models \beta(\C_n)\} = \{\A \in \bg\colon \C_n \not \hookrightarrow \A_*\},
    \]
    since any $\V \subseteq \bg$ is either contained in some $\V_n$ or contains $\V_{FC}$.
    Furthermore, we know from Theorem \ref{main thm} that $|\bigcup_{n \in \ZZ^+} \V_n|= \aleph_0$, hence it follows from the fact that $|\Lambda(\bg)|=2\qq {\aleph_0}$ (see \cite[Thm. 4.16]{Paper1}) that $|\up \vfc| =2\qq {\aleph_0}$.

    Next we show that the outer edges of our lattice are accurately depicted.
    As $\bg$ is congruence-distributive by Theorem \ref{props of bg}.(ii), it is a consequence of Jónsson’s Lemma (see, e.g., \cite[Thm.\ VI.6.10]{Sanka2}) that $\mathbb{V}(K)_{SI} \subseteq \HHH \SSS(K)$ for any finite $K \subseteq \bg^{<\omega}$.
    If we moreover assume that $K\subseteq \bg_{SI}^{<\omega}$, then by recalling that $\bg$ is a semi-simple variety in which every subalgebra of an SI element is also SI (see Theorem \ref{props of bg}.(i)\&.(iii)) we obtain $\mathbb{V}(K)_{SI}= \III \SSS(K)$.
    Using duality, an easy argument now ensures that the subvarieties of $\mathbb{V}(K)$ are in a one-to-one correspondence with the $\leq_p$-antichains of
    \[
    \{\Y^* \colon \Y^* \text{ is a finite co-tree s.t. } \X \twoheadrightarrow \Y \text{ for some } \X^* \in K\},
    \]
    a fact that we will use to prove
    \[
    \down \vfc =\{\mathbb{V}(\C'_0) \prec \mathbb{V}(\C_1) \prec \mathbb{V}(\C'_1) \prec  \mathbb{V}(\C_2) \prec \dots \prec \V_{FC}\}.
    \]
    It follows from \cite[Prop.~3.10]{MARTINS2025103563} that $(\vfc)_{SI}^{<\omega} = \{\C_n^* \colon n \in \ZZ^+\} \cup \{(\C'_n)^{*} \colon n \in \omega\}$.
    Since by definition the variety of the finite combs $\vfc$ has the FMP, this variety is generated by $(\vfc)_{SI}^{<\omega}$, hence any proper subvariety of $\vfc$ is generated by a proper subset of $\{\C_n^* \colon n \in \ZZ^+\} \cup \{(\C'_n)^{*} \colon n \in \omega\}$.
    The proof of the aforementioned result also establishes the existence of a strict chain of surjective bi-p-morphisms
\[
\dots\twoheadrightarrow \C'_n \twoheadrightarrow \C_n \twoheadrightarrow \C'_{n-1} \twoheadrightarrow \C_{n-1} \twoheadrightarrow \dots \twoheadrightarrow \C_1 \twoheadrightarrow \C'_0,
\]
which moreover contains every bi-p-morphic image of each element of the chain.
Using this together with the above characterization of $(\vfc)_{SI}^{<\omega}$, it is easy to see that any variety containing infinitely many duals of elements of the chain must be the whole $\vfc$.
Consequently, the proper subvarieties of $\vfc$ must be generated by finitely many duals of elements of the chain.
By our comments above regarding $\mathbb{V}(K)$ for a finite $K\subseteq \bg_{SI}^{<\omega}$, the properties of the chain yield
\[
    \down \vfc =\{\mathbb{V}(\C'_0) \prec \mathbb{V}(\C_1) \prec \mathbb{V}(\C'_1) \prec  \mathbb{V}(\C_2) \prec \dots \prec \V_{FC}\},
    \]
and thus that the left outer edge of $\Lambda(\bg)$ is faithfully depicted.

    Let us now characterize $\down \V_{\bilc}$, the sublattice of nontrivial varieties which validate the bi-intuitionistic linear calculus (i.e., the bi-intuitionistic logic of chains)
    \[
    \vdash_{\operatorname{\mathsf{bi-LC}}} = \cbipc + (p \to q) \lor (q \to p) + \neg[ (q\gets p) \land (p \gets q)].
    \]
    It is shown in \cite[Thm.~3.10]{Paper1} that $\A \in \V_{\bilc}$ iff the underlying poset of $\A_*$ is a disjoint union chains. 
    This has two immediate consequences. 
    Firstly, we have $(\V_{\bilc})_{SI}^{<\omega}=\{\La_n^* \colon n \in \ZZ^+\}$, where $\La_n$ denotes the $n$-chain (i.e., a linearly ordered poset with $n$ elements).
    
    Secondly, since the only finite comb that can be order embedded into chains is $\C_1 = \La_2$, then $\V_{\bilc} \subseteq \V_2$ holds true (as depicted in Figure \ref{Fig:the lattice}), which by Theorem \ref{The big theorem}.(2)  implies that $\V_{\bilc}$ is locally finite, and thus has the FMP.
    Because of this, and since we clearly have a strict chain of surjective bi-p-morphisms
    \[
    \dots \twoheadrightarrow \La_{n+1}\qq * \twoheadrightarrow \La_{n}\qq * \twoheadrightarrow \La_{n-1}\qq * \twoheadrightarrow \dots \twoheadrightarrow \La_1
    \]
    which moreover contains every bi-p-morphic image of each element of the chain, we can now use a very similar argument to the one used above to infer 
    \[
    \down \V_{\bilc} = \{ \mathbb{V}(\La_1) \prec \mathbb{V}(\La_2) \prec \dots \prec \V_{\bilc} \}.
    \]

Now, recall that by \cite[Cor.~4.5]{Paper1}, the variety $\bg$ has the FMP, hence it is generated by $\{\A \in \bg \colon \A_* \text{ is a finite co-tree}\}$.
Since there are arbitrarily large finite combs, any element of this set is contained in some 
\[
\V_n =\{\A \in \bg\colon \A \models \beta(\C_n)\} = \{\A \in \bg\colon \C_n \not \hookrightarrow \A_*\}.
\]
This makes clear that the only variety of bi-Gödel algebras that contains $\bigcup_{n \in \ZZ^+} \V_n$ is $\bg$ itself.
Furthermore, as $\C_n$ is not contained in $\V_n$, the join $\V_n + \mathbb{V}(\C_n)$ is an immediate successor of $\V_n$.
And we already established in Lemma \ref{lem antichains equivalence} that for $n < m$, we have $\V_n \subsetneq \V_m$.
This concludes the proof that our depiction of the right outer edge of $\Lambda(\bg)$ is accurate.

It remains to show that we correctly partitioned $\bigcup_{n \in \ZZ^+} \V_n$ into $\aleph_0$ pieces, each of size $\aleph_0$.
This will follow from a series of facts.

Firstly, it is proved in \cite[Thm.~4.12]{Paper1} that any finite co-tree, such as $\C_n$, defines a lattice splitting 
\[
\Lambda(\bg) = \up \mathbb{V}(\C_n) \biguplus \down \{ \A \in \bg \colon \A \models \J(\C_n) \},
\]
where $\J(\C_n)$ denotes the Jankov formula of the algebraic dual of $\C_n$.

Secondly, \cite[Cor.~5.22]{Paper1} ensures that for a bi-Gödel algebra $\A$, we have $\A \models \J(\C_n)$ iff $\A \models \beta(\C_n)$.
We thus have
\[
\V_n =\{\A \in \bg\colon \A \models \beta(\C_n)\} =\{\A \in \bg\colon \A \models \J(\C_n)\},
\]
so the lattice splitting above is of the form
\[
\Lambda(\bg) = \up \mathbb{V}(\C_n) \biguplus \down \V_n.
\]
And thirdly, a quick inspection of their structure makes clear that $\C_n \not \hookrightarrow \C'_{n-1}$, hence 
\[
\C'_{n-1} \in \V_n  =\{\A \in \bg\colon \C_n \not \hookrightarrow \A_*\},
\]
and it follows that $\mathbb{V}(\C'_{n-1})\subseteq \V_n$.

Finally, we show $|\down \V_2| = \aleph_0 = | \down \V_n \smallsetminus \down \V_{n-1}|$.
The first equality is immediate from $|\down \V_{\bilc}|=\aleph_0$ and $\V_{\bilc} \subseteq \V_2$.

We now take $n>2$, and prove $|\down \V_n \smallsetminus \down \V_{n-1}| = \aleph_0$ by constructing an infinite strict chain in $\down \V_n \smallsetminus \down \V_{n-1}$, which implies $\aleph_0 \leq | \down \V_n \smallsetminus \down \V_{n-1}| \leq |\V_n|$, and so the desired equality follows from the fact $|\down \V_n|=|\Lambda(\V_n)| \leq \aleph_0$ proved in Theorem \ref{main thm}.

Recall Lemmas \ref{lem structure}\&\ref{lem ord reflecting} and the adjacent discussion.
Let $\Y\in \T_{n-1}\smallsetminus \T_{n-2}$, which exists by Lemma \ref{lem antichains equivalence}.(iii), and for each $t\in \omega$, set
\[
\X_t \coloneqq (\tau(2+t,2),[\Y,\Y]).
\]
By definition we must have $\X_t \in \T_n \smallsetminus \T_{n-1}$, and thus that $\mathbb{V}(\X_t) \in \down \V_n \smallsetminus \down \V_{n-1}$.
Furthermore, it is easy to see that $\X_t \not \twoheadrightarrow \X_r$ and $\X_r \twoheadrightarrow \X_t$ whenever $t < r$.
Clearly, this yields an infinite ascending chain
\[
\mathbb{V}(\X_0) \subsetneq \mathbb{V}(\X_1) \subsetneq \dots,
\]
contained in $\down \V_n \smallsetminus \down \V_{n-1}$, as desired.\end{proof}

\vs

\noindent 
{\bf Acknowledgment}\ \ I was supported by the grant 2023.03419.BD from the Portuguese Foundation for Science and Technology (FCT) and by the proyecto PID$2022$-$141529$NB-C$21$ de investigaci\'on financiado por MICIU/AEI/ 10.13039/501100 011033 y por FEDER, UE. 
I am very thankful to G. Solda and A. Weiermann for introducing me to the theory of better partial orders in an elucidating conversation that made proving Theorem \ref{thm multi proj is bpo} possible.

\bibliographystyle{plain}
\bibliography{biblio}
\end{document}